\documentclass[onefignum,onetabnum]{siamart190516}

\usepackage{todonotes}
\usepackage{bm}
\usepackage{enumerate}

\usepackage{amsfonts}
\usepackage{graphicx}
\usepackage{epstopdf}
\usepackage{algorithmic}
\usepackage{nicefrac}
\usepackage{stmaryrd}
\ifpdf
  \DeclareGraphicsExtensions{.eps,.pdf,.png,.jpg}
\else
  \DeclareGraphicsExtensions{.eps}
\fi
\usepackage{booktabs}

\newsiamremark{remark}{Remark}
\newsiamremark{assumption}{Assumption}
\newsiamremark{hypothesis}{Hypothesis}
\crefname{hypothesis}{Hypothesis}{Hypotheses}
\newsiamthm{claim}{Claim}

\title{Fast-convergent two-level restricted additive Schwarz methods based on optimal local approximation spaces
}

\author{Arne Strehlow\thanks{Institute for Applied Mathematics and Interdisciplinary Center for Scientific Computing, Heidelberg University, Im Neuenheimer Feld 205, Heidelberg 69120, Germany (\email{arne.strehlow@uni-heidelberg.de, r.scheichl@uni-heidelberg.de})}
\and Chupeng Ma\thanks 
    {School of Sciences, Great Bay University, Dongguan 523000, China; Great Bay Institute for Advanced Study, Dongguan 523000, China (\email{chupeng.ma@gbu.edu.cn}).}    
 \and Robert Scheichl\footnotemark[1] 
}

\usepackage{amsopn}

\DeclareMathOperator{\supp}{supp}

\usepackage{graphicx}
\graphicspath{ {./images/} }

\ifpdf
\hypersetup{
}
\fi

\theoremstyle{remark}
\theoremstyle{definition}\newtheorem{example}{Example}
\theoremstyle{condition}

\newcommand{\rate}{\Lambda}
\newcommand{\tforall}{\text{ for all }}

\newcommand{\richIterate}{v}
 \newcommand{\richIterateVector}{{\bf v}}
\newcommand{\normEquivLow}{b_1}
\newcommand{\normEquivHigh}{b_2}

\usepackage[normalem]{ulem}

\begin{document}

\maketitle
%
\begin{abstract}
This paper proposes a two-level restricted additive Schwarz (RAS) method for multiscale PDEs, built on top of a multiscale spectral generalized finite element method (MS-GFEM). The method uses coarse spaces constructed from optimal local approximation spaces, which are based on local eigenproblems posed on (discrete) harmonic spaces. We rigorously prove that the method, used as an iterative solver or as a preconditioner for GMRES, converges at a rate of $\Lambda$, where $\Lambda$ represents the error of the underlying MS-GFEM. The exponential convergence property of MS-GFEM, which is indepdendent of the fine mesh size $h$ even for highly oscillatory and high contrast coefficients, thus guarantees convergence in a few iterations with a small coarse space. We develop the theory in an abstract framework, and demonstrate its generality by applying it to various elliptic problems with highly heterogeneous coefficients, including $H({\rm curl})$ elliptic problems. The performance of the proposed method is systematically evaluated and illustrated via applications to two and three dimensional heterogeneous PDEs, including challenging elasticity problems in realistic composite aero-structures.
\end{abstract}


%
%
%
\section{Introduction} 
Partial differential equations (PDEs) with highly heterogeneous coefficients arise in various practical applications, including subsurface flow in porous media, geodynamic processes, and engineered composite materials. Solving these PDEs directly using standard numerical methods typically necessitates a very fine mesh to resolve the small material heterogeneities represented by the coefficients. This approach results in large-scale, ill-conditioned algebraic equations. Furthermore, important engineering applications -- such as optimal design, sensitivity analysis, and reverse modeling -- require solving problems multiple times for different loadings and/or local structure modifications, rendering direct numerical simulations computationally infeasible. Over the past two decades, considerable efforts have been made to develop efficient numerical methods for multiscale PDE problems. Two notable and closely related classes of methods in this regard are multiscale/localized model reduction methods, which aim to construct effective reduced models, and domain decomposition methods for the iterative solution of discretized PDE problems.


There exists a vast literature on multiscale/localized model reduction methods. While a variety of methods have been developed, recent advancements have focused on coarse space based methods that do not require scale separation of multiscale problems. These methods, often formulated within the variational framework, aim to design a low-dimensional approximation space using problem-adapted, localized basis functions. To obtain provable convergence guarantees, the local bases are typically identified by solving carefully-designed local problems tailored to the underlying differential operators. Once the coarse approximation space is established, the resulting reduced model can be used for multiple simulations, yielding substantial computational savings. These methods include (generalized) multiscale finite element methods ((G)MsFEM) \cite{chung2018constraint,calo2016randomized,efendiev2013generalized,hou1997multiscale,hou1999convergence}, localized orthogonal decomposition (LOD) and its variants \cite{hauck2023super,maalqvist2014localization,freese2024super}, component mode synthesis based methods \cite{hetmaniuk2010special,hetmaniuk2014error,heinlein2015approximate}, rough polyharmonic splines \cite{owhadi2014polyharmonic}, gamblet based numerical homogenization \cite{owhadi2017gamblets,owhadi2019operator}, and ArbiLoMod \cite{buhr2017arbilomod,buhr2018randomized}, to name a few. We refer to \cite{altmann2021numerical,chungmultiscale,efendiev2009multiscale,maalqvist2020numerical} for comprehensive reviews. On the other hand, domain decomposition methods aim at building efficient preconditioners based on solving independent local problems. To achieve scalability and robustness for these preconditioners, it is essential to incorporate a suitable coarse space (coarse problem). In recent years, multiscale coarse spaces, in particular, spectral coarse spaces based on local spectral problems \cite{al2023efficient,agullo2019robust,bastian2022multilevel,ciaramella2022spectral,efendiev2012robust,eikeland2019overlapping,galvis2010domain,galvis2010domain-reduced,gander2015analysis,heinlein2019adaptive,heinlein2018multiscale,heinlein2018adaptive,heinlein2022fully,spillane2014abstract,spillane2021abstract}, have proved effective and become popular in designing coefficient-robust preconditioners for multiscale PDEs. While a key component of both multiscale/localized model reduction and domain decomposition methods is a similarly constructed coarse problem, it plays essentially different roles -- it is designed to approximate the solution space in the former and to provide a stable decomposition in the latter. Due to this difference, coarse problems in domain decomposition methods typically have a much smaller size. 


While multiscale/localized model reduction and domain decomposition methods have proved effective for addressing multiscale PDEs, the repeated solution of large-scale problems remains a challenging task. For multiscale/localized model reduction methods, the reduced models are often low-dimensional, allowing for fast solutions when low approximation accuracy is acceptable. However, for truly large-scale problems and higher accuracy requirements, these models can become considerably large, potentially affecting their efficiency. Conversely, domain decomposition methods typically involve smaller coarse problems, but (as iterative methods) they become computationally expensive when the number of iterations is high. Therefore, it is advantageous to combine these two methods to design a model reduction-based iterative method that features a low-dimensional coarse problem (reduced model) and achieves rapid convergence through superior approximation properties. With a readily solvable coarse problem and a reduced number of iterations per solve, such an iterative method is expected to outperform both the original model reduction method and traditional domain decomposition methods in terms of efficiency for multiple simulations of large-scale problems. This constitutes the motivation for our work.

This work is based on the Multiscale Spectral Generalized Finite Element Method (MS-GFEM) of Babuska and Lipton \cite{babuska2011optimal}, and in particular on the version developed in \cite{ma2021novel,ma2022error}, which is rooted in the GenEO coarse space \cite{spillane2014abstract}. As a special generalized finite element method, it builds local approximation spaces from selected eigenfunctions of GenEO-type local eigenproblems posed on locally harmonic spaces. These local spaces are then glued together using a partition of unity to form the coarse approximation space. Specifically, these local eigenproblems correspond to the singular value decomposition (SVD) of certain compact operators involving the partition of unity, and the resulting local spaces are optimal for approximating locally harmonic functions in the sense of the Kolomogrov $n$-width \cite{pinkus1985n}. For accurate local approximations of a solution, particular functions that locally solve the underlying PDE are integrated into the MS-GFEM approximation. Another key ingredient of MS-GFEM is the oversampling technique -- the local problems are solved in domains slightly larger than the subdomains where the solution is approximated. It was proved that the method, at both continuous and discrete levels, converges exponentially with the number of local degrees of freedom for multiscale problems with general $L^{\infty}$ coefficients. Besides scalar elliptic problems, MS-GFEM has been applied to Helmholtz problems \cite{ma2022exponential}, singularly perturbed problems \cite{ma2024exponential}, and flow problems in mixed formulation \cite{alber2024mixed}. Notably, it has been used to simulate composite aero-structures with a parallel implementation in \cite{benezech2022scalable}, where it was observed that the size of the coarse problem becomes a bottleneck for method scalability in achieving accurate approximations. This observation partly motivates our work. Moreover, a unified, abstract framework of MS-GFEM has been recently established in \cite{ma2023unified}, where a sharper exponential convergence rate was proved, and applications to a number of multiscale PDEs were presented.

In this paper, following the idea of using a smaller coarse problem with a small number of iterations, we first formulate MS-GFEM as a Richardson-type iterative algorithm for scalar elliptic problems. With the coarse problem fixed, the algorithm computes a new approximation by updating the local particular functions in MS-GFEM using local boundary conditions of an obtained approximate solution. We prove that this algorithm converges at a rate of $\Lambda$, where $\Lambda$ is the error of the underlying MS-GFEM. The exponential convergence property of MS-GFEM thus ensures that only a few iterations are needed for the convergence of the algorithm with a small sized coarse problem. The iterative algorithm naturally leads to a hybrid, restricted additive Schwarz (RAS) preconditioner with the MS-GFEM coarse space. We prove that a GMRES iteration with this RAS preconditioner also converges at a rate of $\Lambda$. Indeed, numerical results in \cref{experimentSection} show that our preconditioner incorporates the coarse space optimally -- the preconditioned algorithms converge much faster than those adding the coarse space differently.

By virtue of the abstract theory of MS-GFEM established in \cite{ma2023unified}, we further develop the iterative MS-GFEM method and the resulting hybrid RAS preconditioner for abstract, symmetric positive definite problems, achieving the same theoretical results. Our abstract theory applies to a variety of elliptic multiscale problems, including scalar elliptic problems, elasticity problems, and ${\bm H}({\rm curl})$ elliptic problems. 

Compared to the original MS-GFEM, the iterative algorithms developed herein enable a flexible selection of the coarse problem size, which can be optimized to minimize the solve time; see \cref{experimentSection}. Also, the iterative nature of the algorithms allows for the use of inexact local eigensolves (e.g., using random sampling techniques developed in \cite{calo2016randomized,chen2020random,buhr2018randomized}), thereby reducing the cost of the preconditioner setup. But at present, even with the exact, expensive local eigensolves, the overhead of the setup of the preconditioner can be easily compensated for by direct savings in the solve stage for problems with a large number of right-hand sides. More importantly, the strong theoretical guarantees and broad applicability make them reliable and appealing for challenging multiscale problems, for which effective solvers are currently unavailable.  

We note that there are several recent, closely related works. In \cite{nataf2024coarse}, a GenEO-type coarse space with local spectral problems posed on certain discrete harmonic subspaces was devised and analysed in a fully algebraic setting. This work primarily focuses on designing coarse spaces for general domain decomposition methods, without investigating the convergence rate of the method with respect to the dimension of local spaces. In \cite{hu2024novel}, a two-level RAS preconditioner with a very similarly constructed spectral coarse space was proposed separately for Helmholtz problems, but without using oversampling for defining the method. In \cite{heinlein2022fully}, local spectral problems corresponding to the SVD of the so-called 'transfer' operators defined on local harmonic spaces were used to build fully algebraic spectral coarse spaces. These studies show increasing interest in using MS-GFEM type spectral coarse spaces for domain decomposition methods.

The rest of this paper is structured as follows. In \cref{modelproblem}, we present the MS-GFEM with error estimates for solving a discretized scalar elliptic equation used as the model problem. In \cref{iterativesection}, we present the iterative MS-GFEM algorithm, derive the hybrid RAS preconditioner, and prove upper bounds on the convergence rate of the preconditioned GMRES algorithm. A comparison of this preconditioner with the GenEO preconditioner is also given in this section. In \cref{abstractFramework}, based on the developed abstract theory of MS-GFEM, we establish the iterative method with the same theoretical results for abstract, symmetric positive definite problems, and present several applications. In \cref{experimentSection}, we apply the proposed method to solve a scalar elliptic equation with a high-contrast coefficient and an elasticity problem in composite aero-structures, and systematically evaluate its performance.

\section{Model problem and MS-GFEM} \label{modelproblem}
To simplify the presentation, we will first introduce a model problem, the heterogeneous diffusion equation with homogeneous Dirichlet boundary conditions, and start by presenting the MS-GFEM for this problem. However, the MS-GFEM applies to general elliptic problems, as will be shown in \cref{abstractFramework}.


\subsection{Model problem} \label{continuousmodelproblem}

Let $\Omega\subset \mathbb{R}^{d}$, $d=2,\,3$, be a bounded Lipschitz domain with a polygonal boundary $\Gamma$, and let $A\in \big(L^{\infty}(\Omega)\big)^{d\times d}$ be a matrix-valued function that is pointwise symmetric and satisfies the uniform elliptic condition: there exist constants $0< \alpha < \beta<\infty$ such that
$$
\alpha |\zeta|^{2} \leq A(x)\zeta\cdot\zeta \leq \beta  |\zeta|^{2},\quad \forall {\zeta}\in \mathbb{R}^{d},\quad x\in\Omega.
$$
Given a linear functional $\ell$ on $H_{0}^{1}(\Omega)$, we seek $u \in H_{0}^{1}(\Omega)$ such that 
\begin{equation}\label{model_problem_continuous}
a(u,v):=\int_{\Omega}(A\nabla u)\cdot\nabla v \,dx = \ell(v) \quad \tforall v \in H_{0}^{1}(\Omega).   
\end{equation}
Given the assumptions on the coefficient $A$, the bilinear form $a(\cdot,\cdot)$ is bounded and coercive, and hence, problem \cref{model_problem_continuous} has a unique solution by the Lax-Milgram theorem.

Now we consider the standard finite element approximation of problem \cref{model_problem_continuous}. Let $\{\tau_h\}$ be a family of shape-regular triangulations of $\Omega$ into triangles (tetrahedra), where $h = \max_{K \in \tau_h}\text{diam}(K)$. We assume that the mesh-size $h$ is small enough such that the fine-scale details of the coefficient $A$ are resolved by $\tau_{h}$. Let $V_h \subset H_{0}^1(\Omega)$ be a finite element subspace consisting of continuous piecewise polynomials. The discretization of problem \cref{model_problem_continuous} is given by: find $u_{h} \in V_{h}$ such that
\begin{equation}\label{model_problem_discrete}
a(u_{h},v) = \ell(v) \quad \tforall v \in V_{h}.   
\end{equation}
Let $\{\psi_{k}\}_{k=1}^{n}$ be a basis for $V_{h}$ with $n:={\rm dim}(V_{h})$. Then, problem \cref{model_problem_discrete} can be written as a linear system:
\begin{equation}\label{linear_system}
 {\bf A}{\bf u} = {\bf f},   
\end{equation}
where ${\bf A} = (a_{ij})\in \mathbb{R}^{n\times n}$ and ${\bf f} = (f_i)\in \mathbb{R}^{n}$ are given by $a_{ij} = a(\psi_i,\psi_j)$ and $f_i = \ell(\psi_i)$ for $i,j = 1,\cdots,n$. In practical applications, the linear system \cref{linear_system} often contains a huge number of degrees of freedom. Moreover, the presence of coefficient heterogeneities makes the matrix ${\bf A}$ notoriously ill-conditioned. Hence, efficient model order reduction or preconditioning techniques are needed for solving problem \cref{model_problem_discrete}. 

For later use, let us introduce some notation. For any subdomain $\omega\subset \Omega$, we define the local FE spaces:
\begin{gather*}
   V_{h}(\omega) = \{v|_{\omega}\,:\,v\in V_{h} \}, \quad 
  V_{h,0}(\omega) = \{v\in V_{h}(\omega)\,:\,\supp(v)\subset \overline{\omega}\}, 
\end{gather*}
and the local bilinear form $a_{\omega}(\cdot,\cdot)$:
\begin{equation}
a_{\omega}(u,v):=\int_{\omega}(A\nabla u)\cdot\nabla v \,dx\quad \text{for all}\;\; u,v\in H^{1}(\omega).
\end{equation}
Moreover, we define 
\begin{equation*}
 \Vert u \Vert_{a,\omega} := \sqrt{a_{\omega}(u,u)} \quad \text{for all}\;\; u\in H^{1}(\omega),
\end{equation*}
and when $\omega=\Omega$, we simply write $\Vert u \Vert_{a}$.

\subsection{GFEM}
In this subsection, we briefly present the Generalized Finite Element Method (GFEM) upon which the MS-GFEM is built. The classical GFEM was formulated at the continuous level (typically in the $H^{1}$ setting) for directly discretizing PDE problems. Here, nevertheless, it will be adapted to the FE setting for solving discretized problems. Similar techniques were widely used in localized model reduction and domain decomposition methods \cite{efendiev2011multiscale,spillane2014abstract}. 

Let $\{ \omega_i\}_{i=1}^{M}$ with $\Omega = \cup_{i=1}^{M}\omega_i$ be an overlapping decomposition of the domain $\Omega$. We assume that all the subdomains $\omega_i$ are resolved by $\tau_{h}$. Let $\xi$ denote the maximum number of the subdomains that overlap at any given point, i.e.,
\begin{equation}\label{coloring-constant}
\xi:= \max_{x \in \Omega}{\left( \text{card}\{i \ | \ x \in \omega_i \} \right)}. 
\end{equation}
A key ingredient of the GFEM is a partition of unity. Let $\{\chi_{i}\}_{i=1}^{M}$ be a partition of unity subordinate to the open cover $\{ \omega_i\}_{i=1}^{M}$ that satisfies:
\begin{equation}\label{POU}
\begin{array}{lll}
{\displaystyle \supp (\chi_i)\subset \overline{\omega_i}, \quad 0\leq \chi_{i}\leq 1,\quad \sum_{i=1}^{M}\chi_{i} =1 \;\;\text{on} \;\,\Omega,}\\[4mm]
{\displaystyle \chi_{i}\in W^{1,\infty}(\omega_{i}),\;\;\Vert\nabla \chi_{i} \Vert_{L^{\infty}(\omega_i)} \leq \frac{C_{\chi}}{\mathrm{diam}\,(\omega_{i})}.}
\end{array}
\end{equation}
To adapt the GFEM to the FE setting, we need a mesh-adapted partition of unity. Let $\{\chi_{h,i}\}_{i=1}^{M}$ be a family of operators defined by
\begin{align}
    \chi_{h,i}: V_{h}(\omega_i) &\rightarrow V_{h,0}(\omega_i), \label{POU-operators} \\
    v &\mapsto I_h(\chi_i v) \nonumber,
\end{align}
where $I_h: H_{0}^1(\Omega) \rightarrow V_h$ is the Lagrange interpolation operator. Clearly, the operators $\chi_{h,i}$ have the partition of unity property:
$$
     \sum\limits_{i = 1}^{M} \chi_{i,h}(v|_{\omega_i}) = v \quad \text{for all}\;\; v\in V_{h}.
 $$
In what follows, $\chi_{h,i}$, $i=1,\cdots,M$, are referred to as the partition of unity operators.

The basic idea of the GFEM is to construct the global ansatz space by gluing carefully designed local approximation spaces together by the partition of unity. On each subdomain $\omega_i$, let a local particular function $u_{i}^{p}\in V_{h}(\omega_i)$ and a local approximation space $S_{m_i}(\omega_i) \subset V_{h}(\omega_i)$ of dimension $m_i$ be given, which will be defined later. Then, the global particular function $u^{p}$ and the global approximation space $S_{m}(\Omega)$ of dimension $m = m_{1}+\cdots+m_{M}$ are defined by 
$$
    u^p = \sum\limits_{i=1}^M \chi_{h,i}(u_{i}^p)\in V_{h}, \quad S_{m}(\Omega) = \left\{ \sum\limits_{i=1}^M \chi_{h,i}( \phi_i) \ | \ \phi_i \in S_{m_i}(\omega_i) \right\} \subset V_{h}.
$$
The GFEM approximation of the fine-scale FE problem \cref{model_problem_discrete} is then defined by: find $u_{h}^G \in u^p + S_{m}(\Omega)$ such that
$$
    a(u_{h}^G, v) = \ell(v) \quad \tforall v \in S_{m}(\Omega).
$$
By Cea's lemma, the GFEM approximation $u_{h}^G$ is the best approximation of $u_h$ in $u^p + S_{m}(\Omega)$, i.e.
\begin{equation}\label{best-approx}
 \lVert u_h - u_{h}^G \rVert_a = \inf\limits_{\phi \in u^p + S_{m}(\Omega)} \lVert u_h - \phi \rVert_a. 
\end{equation}
Hence, the accuracy of the GFEM is determined by the quality of the global approximation, which, in turn, hinges on the accuracy of the local approximations. Indeed, we have the following fundamental approximation theorem, which is a variant of the classical approximation theorem of the GFEM (see \cite[Theorem 2.1]{melenk1996partition}). The difference arises from the local approximation strategy -- we aim at approximating $\chi_{h,i} ({u_h}|_{\omega_i})$ instead of ${u_h}|_{\omega_i}$ alone as in the classical GFEM.

\begin{theorem} \label{gfem}
    Let $v \in V_{h}$. Assuming that for each $i = 1,...,M$, 
    $$
        \inf\limits_{\phi \in u_{i}^p + S_{m_i}(\omega_i)} \big\lVert \chi_{h,i} (v|_{\omega_i}- \phi) \big\rVert_{a, \omega_i} \leq e_i,
    $$
    then,
    $$
        \inf\limits_{\phi \in u^p + S_{m}(\Omega)} \lVert v - \phi \rVert_a \leq \Big( \xi\sum_{i=1}^{M} e^{2}_i \Big)^{1/2},
    $$
    where $\xi$ is defined by \cref{coloring-constant}.
\end{theorem}

Based on \cref{best-approx} and \cref{gfem}, it is clear that the core of the GFEM lies in a judicious selection of the local particular functions $u_{i}^{p}$ and the local approximation spaces $S_{m_i}(\omega_i)$ such that $\chi_{h,i} ({u_h}|_{\omega_i})$ can be well approximated in $\chi_{h,i} \big(u_{i}^{p} +  S_{m_i}(\omega_i)\big)$. A particular local construction with exponential accuracy will be given below.

\subsection{MS-GFEM}\label{msgfem}
Within the MS-GFEM, the local particular functions $u_{i}^{p}$ are defined as local solutions of the underlying PDE, and the local approximation spaces $S_{m_i}(\omega_i)$ are spanned by selected singular vectors of certain compact operators involving the partition of unity operators. With such local approximations, the method achieves exponential convergence with respect to the number $m_i$ of local degrees of freedom. To define $u_{i}^{p}$ and $S_{m_i}(\omega_i)$ precisely, we introduce a set of oversampling domains $\omega_i^* \supset \omega_i$, $1\leq i\leq M$, which we assume to be resolved by $\tau_h$.


On each oversampling domain $\omega_i^{\ast}$, we consider the following local problem: Find $\varphi_{i} \in V_{h,0}(\omega_i^*)$ such that 
$$
    a_{\omega_i^{\ast}}(\varphi_{i},v) = \ell(v) \quad \tforall v \in V_{h,0}(\omega_i^*).
$$
The local particular function $u_{i}^{p}$ is then defined by 
\begin{equation}\label{local-par-sol}
u_{i}^{p} = \varphi_{i}|_{\omega_i}.     
\end{equation}
To define the local approximation space, we first introduce the $a$-harmonic subspace of $V_{h}(\omega_i^*)$:
$$
    V_{h,A}(\omega_i^*) := \big\{ v \in V_{h}(\omega_i^*) \ | \ a_{\omega_i^*}(v,w) = 0 \quad \forall w \in V_{h,0}(\omega_i^*) \big\},
$$
i.e. the $a$-orthogonal complement of $V_{h,0}(\omega_i^*)$ in $V_{h}(\omega_i^*)$. We note that $u_{h}|_{\omega_i^{\ast}} - \varphi_{i}\in V_{h,A}(\omega_i^*)$. Thus, we consider the following local eigenproblem: Find $\lambda_i \in \mathbb{R}\cup \{+\infty\}$ and $\phi_{i} \in V_{h,A}(\omega_i^*)$ such that
\begin{equation} \label{eigenproblemdiscrete}
   a_{\omega_i}\big(\chi_{h,i} (\phi_{i}|_{\omega_i}), \,\chi_{h,i} (v|_{\omega_i}) \big)=\lambda_{i} \,a_{\omega_i^{\ast}}(\phi_{i}, v) \quad \tforall v \in V_{h,A}(\omega_i^*).
\end{equation}
Then, the desired local approximation space $S_{m_i}(\omega_i) \subset V_{h,A}(\omega_i)$ is defined by 
\begin{equation}\label{local-appro-space}
S_{m_i}(\omega_i) = {\rm span} \big\{\phi_{i,1}|_{\omega_i}, \cdots, \phi_{i,m_i}|_{\omega_i} \big\},    
\end{equation}
where $\phi_{i,j}$ denotes the $j$-th eigenfunction of problem \cref{eigenproblemdiscrete} (with eigenvalues arranged in decreasing order). We note that the bases of $S_{m_i}(\omega_i)$ are the (right) singular vectors of the operator
\begin{equation*}
    P_{i}: V_{h,A}(\omega_i^*) \rightarrow V_{h,0}(\omega_i) \quad  v \mapsto \chi_{h,i}(v|_{\omega_i}).
\end{equation*}
The optimality of the space $S_{m_i}(\omega_i)$ for approximating locally a-harmonic functions was demonstrated in \cite{ma2022error} by means of the concept of the Kolmogorov $n$-width.

Let $\lambda_{i,m_i+1}$ be the $(m_i+1)$-th eigenvalue of problem \cref{eigenproblemdiscrete}, i.e., the eigenvalue corresponding to the first eigenfunction not included in $S_{m_i}(\omega_i)$. The choice for $u_{i}^{p}$ and $S_{m_i}(\omega_i)$ leads to the following local approximation property \cite[Theorem 3.3]{ma2022error}.
\begin{lemma}\label{local-appro-property}
Let $u_{h}$ be the solution of problem \cref{model_problem_discrete}, and let $u_{i}^{p}$ and $S_{m_i}(\omega_i)$ be defined by \cref{local-par-sol} and \cref{local-appro-space}, respectively. Then,
    $$
        \inf\limits_{\phi \in u_{i}^{p} + S_{m_i}(\omega_i)} \lVert \chi_{h,i} (u_h|_{\omega_i} - \phi) \rVert_{a,{\omega_i}} \leq \lambda_{i,m_i+1}^{1/2} \lVert u_h \rVert_{a,\omega_i^*}.
    $$

\end{lemma}
The following global approximation result is a direct consequence of \cref{local-appro-property}, \cref{gfem}, and the best approximation property \cref{best-approx}.
\begin{corollary} \label{msgfemapprox}
Let the local particular solutions $u_{i}^{p}$ and the local approximation spaces $S_{m_i}(\omega_i)$ be defined by \cref{local-par-sol} and \cref{local-appro-space}, respectively. Then,
    $$
        \big\lVert u_h - u_h^{G} \big\rVert_a \leq \Lambda \lVert u_h \rVert_{a}
    $$
    with the constant $\Lambda$ given by
    \begin{equation}\label{Lamdba}
     \Lambda := \Big( \xi \xi^{\ast}\max\limits_{i=1,..,M} \lambda_{i,m_i+1} \Big)^{1/2},   
    \end{equation}
where $\xi$ is defined by \cref{coloring-constant}, and $\xi^{\ast}$ is defined similarly for the oversampling domains.    
\end{corollary}

\Cref{msgfemapprox} indicates that a rapid decay of the eigenvalues $\lambda_{i,j}$ (with $j$) is central to the efficiency of the MS-GFEM: the faster $\lambda_{i,j}$ decay, the fewer eigenfunctions are needed for a given error tolerance, and thus the smaller the global coarse problem is. The following theorem shows the exponential decay of the eigenvalues as expected.

\begin{theorem}\cite[Theorem 4.8]{ma2022error} \label{expcont}
Let $\delta_{i}^{\ast}= {\rm dist}(\omega_i, \partial \omega_i^{\ast}\setminus \partial \Omega)>0$. There exist $k_{i},\,b_i>0$ independent of $h$, such that for any $k>k_{i}$, if $h\leq \delta_{i}^{\ast}/(4b_{i}k^{1/(d+1)})$, then
\begin{equation*}
\lambda^{1/2}_{i,k} \leq (1 + C_{\chi})e^{-b_{i}k^{{1}/{(d+1)}}},
\end{equation*}
where $C_{\chi}$ is the constant given in \cref{POU}.
\end{theorem}

\begin{remark}
The constants $k_i$ and $b_i$ above were given explicitly in \cite{ma2022error}. In particular, $k_i$ decreases and $b_i$ increases with the growing oversampling size $\delta^{\ast}_{i}$. Moreover, the conditions on $k$ and $h$ are not essential, and can indeed be removed; see \cite{chen2023exponentially,angleitner2023exponential}.
\end{remark}

\section{MS-GFEM based two-level RAS methods} \label{iterativesection}
In this section we will develop two-level restricted additive Schwarz (RAS) methods based on MS-GFEM for solving the linear system \cref{linear_system}. It will be proved that the methods converge at least at a rate of $\Lambda$ (defined in \cref{Lamdba}).



\subsection{Iterative MS-GFEM}\label{iterativemsgfem}
We start by formulating the MS-GFEM presented in the preceding section, which yields the approximate solution in one shot, as an iterative method. Before doing this, let us describe the idea behind the method. Note that the MS-GFEM approximation can be written as $u_{h}^{G} = u^{p} + u^{S}$, where $u^{S}\in S_{m}(\Omega)$ satisfies
\begin{equation*}
a(u^{S}, v) = a(u_{h}-u^{p},v)\quad \tforall v\in S_{m}(\Omega),
\end{equation*}
i.e., the $a$-orthogonal projection of $u_{h}-u^{p}$ onto the coarse space $S_{m}(\Omega)$. Therefore, the MS-GFEM can be viewed as a coarse space correction to the naive approximation $u^{p}$ obtained by pasting the local solutions together. In general, $u^{p}$ is a poor approximation of $u_{h}$ due to the inaccurate (zero) boundary conditions used for the local solutions. With $u_{h}^{G}$ in hand, a natural idea to obtain a better approximation is to use the local boundary conditions of $u_{h}^{G}$ to update the local solutions, and then compute the coarse space correction. This observation motivates the iterative method below.

To formulate the iterative method in operator language, we need some notation. Recall the local FE spaces $V_{h,0}(\omega_i^*)$ $(1\leq i\leq M)$ and the global approximation space $S_{m}(\Omega)$. We denote by $\pi_i$ and $\pi_S$ the $a$-orthogonal projections of $V_{h}$ onto $V_{h,0}(\omega_i^*)$ and $S_{m}(\Omega)$, respectively, i.e., for any $v\in V_{h}$,
\begin{align*}
\begin{split}
a_{\omega_i^{\ast}}(\pi_i(v),w) &= a(v,w) \quad \tforall w \in V_{h,0}(\omega_i^*),\\[1mm]
a(\pi_S(v),w) &= a(v,w) \quad \tforall w \in S_{m}(\Omega).
\end{split}
\end{align*}
Moreover, we define the operators $\widetilde{\chi}_{h,i}: V_{h,0}(\omega_i^*)\rightarrow V_{h,0}(\omega_i^*)$ ($1\leq i\leq M$) by 
$$\widetilde{\chi}_{h,i}(v) = {\chi}_{h,i}(v|_{\omega_i}),$$
where $\chi_{h,i}$ is defined by \cref{POU-operators}, and we have identified $V_{h,0}(\omega_i)$ as a subspace of $V_{h,0}(\omega^{\ast}_i)$. With these operators we can define the MS-GFEM map as follows.
\begin{definition}\label{MSGFEM-map}
    We call the map $G:V_{h} \rightarrow V_{h}$ given by
    $$
        G(v) = \sum\limits_{i=1}^M \widetilde{\chi}_{h,i} \pi_i(v) + \pi_S \Big( v - \sum\limits_{i=1}^M \widetilde{\chi}_{h,i} \pi_i(v) \Big)
    $$
    the MS-GFEM map.  
\end{definition}
The map $G$ is linear and depends on $a(\cdot,\cdot)$, as well as on the choices of the $\omega_i$, $\omega_i^*$, $\chi_i$ and $m_i$, all of which remain fixed throughout this section. Note that for any $v\in V_{h}$, $G(v)$ is exactly the MS-GFEM approximation for the fine-scale problem \cref{model_problem_discrete} with the right-hand side $\ell(\cdot):=a(v, \cdot)$. Hence, we have the following estimate by \cref{msgfemapprox}.
\begin{lemma} \label{cor1}
 Let $\Lambda$ be defined by \cref{Lamdba}. For any $v \in V_{h}$,
 \begin{equation*}
\lVert v - G(v) \rVert_a \leq \rate \lVert v \rVert_{a}.    
 \end{equation*}

\end{lemma}
Now we can define the iterative method by means of the map $G$.
\begin{definition}[Iterative MS-GFEM] \label{iteration}
Let $u_{h}$ be the solution of problem \cref{model_problem_discrete}. Given an initial guess $\richIterate^0 \in V_{h}$, let the sequence $\{\richIterate^j\}_{j \in \mathbb{N}}$ be constructed by 
    \begin{equation} \label{richardson-iteration-equation}
        \richIterate^{j+1} := \richIterate^j + G\left( u_h - \richIterate^j \right),\quad  j=0,1\cdots.
    \end{equation}
We will call this algorithm the MS-GFEM iteration in the following.
\end{definition}

While $u_{h}$ is involved in the algorithm \cref{richardson-iteration-equation}, it is indeed not needed for the iteration as we can use the equation $a(u_{h},v)= \ell (v)\;\;\forall v\in V_{h}$. The iterative MS-GFEM is a RAS type algorithm (see \cite[Chapter 1]{dolean2015introduction}) with a coarse-space correction in each iteration. A key difference between our method and the traditional RAS method is that the local problems here are solved on the enlarged subdomains, which, combined with the specially chosen coarse space, leads to a fast convergence. A simple application of \cref{cor1} shows that the iteration converges at a rate of $\rate$.
\begin{proposition} \label{cor2}
Let the sequence $\{\richIterate^j\}_{j \in \mathbb{N}}$ be defined by \cref{iteration}. Then,  
    \begin{equation*}
    \lVert \richIterate^{j+1} - u_h \rVert_a \leq \rate \lVert \richIterate^j - u_h \rVert_{a},\quad j=0,1,\cdots.  
    \end{equation*}
\end{proposition}
   \begin{proof}
   Use the relation \cref{richardson-iteration-equation} and \cref{cor1}.
    \end{proof}

Next we give the matrix formulation for the MS-GFEM iteration. For each $1\leq i\leq M$, let $\widetilde{\bf R}_i^T$ be the matrix representation of the zero extension operator $V_{h,0}(\omega_i^*) \rightarrow V_{h}$ with respect to the basis $\{\psi_{k}\}_{k=1}^{n}$ of $V_{h}$. Here we use the symbol $\;\widetilde{}\;$ to distinguish them from matrix representations of the usual extension operators $V_{h,0}(\omega_i)\rightarrow V_{h}$. Similarly, let ${\bf R}^{T}_S$ be the matrix representation of the natural embedding $S_{m}(\Omega) \rightarrow V_{h}$ with respect to the chosen basis of $S_{m}(\Omega)$. Using these matrices and the stiffness matrix $\bf A$, we can write the projections $\pi_i$ and $\pi_S$ in matrix form as follows: 
$$
    \boldsymbol{\pi}_i = \widetilde{\bf A}_i^{-1} \widetilde{\bf R}_i {\bf A} \quad \text{and} \quad \boldsymbol{\pi}_S = {\bf A}_S ^{-1}{\bf R}_S {\bf A},
$$
with 
$$
    \widetilde{\bf A}_i: = \widetilde{\bf R}_i {\bf A} \widetilde{\bf R}_i^T \quad \text{and} \quad {\bf A}_S: = {\bf R}_S {\bf A} {\bf R}_S^T.
$$
Then the matrix representation of the MS-GFEM map $G$ is given by
\begin{align*}
  {\bf G} 
    &= \left(  \sum_{i=1}^{M} \widetilde{\bf R}_i^T \widetilde{\boldsymbol{\chi}}_i \widetilde{\bf A}_i^{-1} \widetilde{\bf R}_i {\bf A} \right)
    + 
    \big({\bf R}_S^T {\bf A}_S^{-1} {\bf R}_S {\bf A} \big) 
    \left( {\bf I} - \sum_{i=1}^{M} \widetilde{\bf R}_i^T \widetilde{\boldsymbol{\chi}}_i \widetilde{\bf A}_i^{-1} \widetilde{\bf R}_i {\bf A}  \right),
\end{align*}
where ${\bf I}$ denotes the identity matrix, and $\widetilde{\boldsymbol{\chi}}_i$ is the matrix representation of the operator $\widetilde{\chi}_{h,i}$. Note that the matrix ${\bf G}$ can be written as ${\bf G} = {\bf B}{\bf A}$, with the two-level hybrid RAS preconditioner
\begin{equation} \label{preconditioner}
    {\bf B} = \left(  \sum_{i=1}^{M} \widetilde{\bf R}_i^T \widetilde{\boldsymbol{\chi}}_i \widetilde{\bf A}_i^{-1} \widetilde{\bf R}_i \right)
    + 
    \big({\bf R}_S^T {\bf A}_S^{-1} {\bf R}_S \big) 
    \left( {\bf I} - {\bf A}\sum_{i=1}^{M} \widetilde{\bf R}_i^T \widetilde{\boldsymbol{\chi}}_i \widetilde{\bf A}_i^{-1} \widetilde{\bf R}_i\right).
\end{equation}
The matrix formulation of the MS-GFEM iteration \cref{richardson-iteration-equation} is now given by
\begin{equation} \label{matrix-form-iteration}
    \richIterateVector^{j+1} := \richIterateVector^{j} + {\bf BA}({\bf u} - \richIterateVector^{j}) = \richIterateVector^{j} + {\bf B}({\bf f} - {\bf A}\richIterateVector^{j}).
\end{equation}
In practice, the algorithm \cref{matrix-form-iteration} is performed in two steps:
\begin{align*}
\begin{split}
\richIterateVector^{j+\frac12} &=  \richIterateVector^{j} + \sum_{i=1}^{M} \widetilde{\bf R}_i^T \widetilde{\boldsymbol{\chi}}_i \widetilde{\bf A}_i^{-1} \widetilde{\bf R}_i ({\bf f} - {\bf A}\richIterateVector^{j}), \\ 
\richIterateVector^{j+1} &=  \richIterateVector^{j+\frac12} + {\bf R}_S^T {\bf A}_S^{-1} {\bf R}_S ({\bf f} - {\bf A}\richIterateVector^{j+\frac12}). 
\end{split}
\end{align*}
In practice, the preconditioner ${\bf B}$ is often used in conjunction with a Krylov subspace method to obtain faster convergence. This will be discussed in the next subsection.

\subsection{Preconditioned GMRES algorithm} \label{preconditionersection}
We consider the solution of the preconditioned system
\begin{equation} \label{preconditioned-system}
   {\bf BAu}  = {\bf Bf}
\end{equation}
with ${\bf B}$ defined by \cref{preconditioner}. As usual RAS preconditioners, ${\bf B}$ is a nonsymmetric preconditioner and thus we use GMRES to solve the system \cref{preconditioned-system}. For completeness, we briefly describe the basic idea of the algorithm below. Let $b(\cdot,\cdot)$ be an inner product on $\mathbb{R}^{n}$ and $\Vert\cdot\Vert_{b}$ the induced norm. Given an initial vector ${\bf u}^{0}\in \mathbb{R}^{n}$, the GMRES algorithm applied to \cref{preconditioned-system} seeks a sequence of approximate solutions ${\bf u}^{j}$, $j=1,2\cdots$, by solving the least-square problems 
\begin{equation}\label{GMRES-minimization}
\min_{{\bf v} \in {\bf u}^{0}+\mathcal{K}_j} \lVert {\bf BA}({\bf u} - {\bf v}) \rVert_b    
\end{equation}
with the Krylov subspaces
$$\mathcal{K}_j:= \text{span} \big\{ {\bf BA}({\bf u} - {\bf u}^{0}), \,({\bf BA})^{2}({\bf u} - {\bf u}^{0}),...,({\bf BA})^{j}({\bf u} - {\bf u}^{0}) \big\}.$$

Thanks to the minimization property \cref{GMRES-minimization}, GMRES approximates the exact solution ${\bf u}$ at least as well as the iterative MS-GFEM algorithm, with the residual measured in the $b$-norm. More precisely, we have
\begin{proposition} \label{gmresbetter}
Let $\{ {\bf u}^{j} \}_{j\in \mathbb{N}}$ and $\{ \richIterateVector^{j} \}_{j\in \mathbb{N}}$ be the sequences of approximate solutions generated by GMRES and by the iterative MS-GFEM with the same initial guess, respectively. Then,
    $$
        \lVert {\bf BA} ({\bf u} - {\bf u}^{j}) \rVert_b \leq  \lVert {\bf BA} ({\bf u} - \richIterateVector^j) \rVert_b,\quad j=1,2\cdots.
    $$
\end{proposition}
\begin{proof}
By definition, ${\bf u}^{j}$ minimizes $\lVert {\bf BA}({\bf u} - {\bf v}) \rVert_b$ over all ${\bf v} \in {\bf u}^0 + \mathcal{K}_j$. On the other hand, $\richIterateVector_j$ lies in $ {\bf u}^0 + \mathcal{K}_j$ as can be shown by an easy induction. 
    \end{proof}

In the rest of this subsection, we derive the rate of convergence for the GMRES applied to \cref{preconditioned-system}. For any ${\bf v}\in \mathbb{R}^{n}$ and ${\bf W}\in \mathbb{R}^{n\times n}$, we define the $a$-norm of ${\bf v}$ as $\Vert {\bf v}\Vert_{a}:= \sqrt{{\bf v}^{T}{\bf A}{\bf v}}$, and denote by $\Vert {\bf W}\Vert_{a}$ the induced matrix norm. In general, the $a$-norm is different from the $b$-norm applied within GMRES which is often the Euclidean norm. In order to switch between these two norms in the analysis, we need the following assumption concerning their equivalence.

\begin{assumption}\label{norm-equiv}
     There exist constants $\normEquivLow, \normEquivHigh > 0$ such that for all ${\bf v} \in \mathbb{R}^n$,
    $$
        \normEquivLow \lVert {\bf v} \rVert _b \leq \lVert {\bf v} \rVert _a \leq \normEquivHigh \lVert {\bf v} \rVert _b.
    $$    
\end{assumption}

\begin{remark} \label{lagrangeBasis}
Assume that the family of meshes $\{\tau_h\}$ is quasiuniform, and that $\Vert\cdot\Vert_{b}$ is the Euclidean norm. It can be shown, by using inverse and norm equivalence estimates, that for all ${\bf v}\in \mathbb{R}^{n}$,
    $$
        \underline{C} \beta^{-1/2} h^{-d/2+1} \lVert {\bf v} \rVert_{b} \leq \lVert {\bf v} \rVert _{a} \leq \overline{C} \alpha^{-1/2} h^{-d/2}  \lVert {\bf v} \rVert _{b},
    $$
where constants $\underline{C}, \overline{C} > 0$ are independent of all parameters.
\end{remark}



The following lemma gives an upper bound on the condition number (in the $a$-norm) of the preconditioned matrix ${\bf BA}$, which is useful in the analysis below.
\begin{lemma} \label{conditionnumber}
   Let $\Lambda$ be defined by \cref{Lamdba}, and ${\bf B}$ as defined in \cref{preconditioner}. Then,
    \begin{equation} \label{condition-estimate}
        \lVert {\bf I} - {\bf BA} \rVert_a \leq \rate. 
    \end{equation}
    Moreover, if $\Lambda <1$, then
    \begin{equation}\label{cn-bound}
      \lVert {\bf BA} \rVert_a \lVert ({\bf BA})^{-1} \rVert_a  \leq \frac{1+\Lambda}{1-\Lambda}.
    \end{equation}
\end{lemma}    
    \begin{proof}
        From \cref{cor1} we deduce that for all ${\bf v}\in \mathbb{R}^{n}$,
        $
            \lVert ({\bf I} - {\bf BA}){\bf v} \rVert_a \leq \rate \lVert {\bf v} \rVert_a,
        $
        which immediately yields
        $
           \lVert {\bf I} - {\bf BA} \rVert_a \leq \rate,
        $
        and thus $\lVert {\bf BA} \rVert_a \leq 1+\Lambda$. To bound  $\lVert ({\bf BA})^{-1} \rVert_a$, we notice that if $\Lambda <1$, then the Neumann series
        \begin{equation*}
            \sum_{i=0}^{\infty} ({\bf I} - {\bf BA})^{i}
        \end{equation*}
converges, and it holds that
        \begin{equation*}
           ({\bf BA})^{-1}=\sum_{i=0}^{\infty} ({\bf I} - {\bf BA})^{i}.
        \end{equation*}
        Therefore, we have $\lVert ({\bf BA})^{-1} \rVert_a \leq (1-\lVert {\bf I} - {\bf BA} \rVert_a)^{-1}\leq (1-\Lambda)^{-1}$, which completes the proof of \cref{cn-bound}.
    \end{proof}

Now we are ready to establish the convergence rate for GMRES applied to \cref{preconditioned-system}.

\begin{theorem}\label{gmresconvergence}
Let $\{ {\bf u}^{j} \}_{j\in \mathbb{N}}$ be the sequence of approximate solutions generated by GMRES, and let $\Lambda$ be defined by \cref{Lamdba}. Then, if $\Lambda<1$, 
    $$
        \lVert {\bf BA}({\bf u} - {\bf u}^{j}) \rVert_b \leq \rate^j \left( \frac{1+\rate}{1-\rate}\right)\frac{b_2}{b_1}  \lVert {\bf BA}({\bf u} - {\bf u}^{0}) \rVert_b,\quad j=1,2\cdots,
    $$
where $\normEquivLow, \normEquivHigh > 0$ are given in \cref{norm-equiv}.   
    \begin{proof}
    Using \cref{gmresbetter} and \cref{norm-equiv} yields
        \begin{align*}
            \lVert {\bf BA}({\bf u} - {\bf u}^{j}) \rVert_b 
            \leq \lVert {\bf BA}({\bf u} - {\bf v}^{j}) \rVert_b
            \leq \frac{1}{b_1}\lVert {\bf BA} \rVert_a \lVert {\bf u} - {\bf v}^{j} \rVert_a.
            \end{align*}
Combining the convergence estimate for the iterative MS-GFEM in \cref{cor2}, the initial guess $\richIterateVector^0 = {\bf u}^0$, and \cref{norm-equiv}, we obtain
       \begin{align*}
           \lVert {\bf u} - \richIterateVector^j \rVert_a
            &\leq  \rate^j \lVert {\bf u} - \richIterateVector^0 \rVert_a\leq  \rate^j \lVert ({\bf BA})^{-1} \rVert_a \lVert {\bf BA}({\bf u} - {\bf u}^0) \rVert_a \\
            & \leq  b_{2}\rate^j \lVert ({\bf BA})^{-1}\rVert_a\lVert {\bf BA}({\bf u} - {\bf u}^0) \rVert_b.
        \end{align*}
Combining the two estimates above and using \cref{conditionnumber} complete the proof. 
    \end{proof}
\end{theorem} 

The following convergence estimate follows from \cref{gmresconvergence} and \cref{lagrangeBasis}.  
\begin{corollary}
Let $\Vert\cdot\Vert_{b}$ be the Euclidean norm, and let $\Lambda < 1$. Suppose that the family of meshes $\{\tau_h\}$ is quasiuniform. Then, GMRES applied to the preconditioned system \cref{preconditioned-system} satisfies the bound
\begin{equation*}
        \lVert {\bf BA}({\bf u} - {\bf u}^{j+k}) \rVert_{b} \leq \rate^j \lVert  {\bf BA}({\bf u} - {\bf u}^{0}) \rVert_{b}
\end{equation*}
with an integer $k$ that grows at most proportionally to  $\log(\beta/\alpha) + \log(h^{-1})$.
\end{corollary}

\subsection{Comparison to GenEO}\label{comparison-to-geneo}
The coarse space within MS-GFEM was motivated by the GenEO coarse space \cite{spillane2014abstract}. GenEO is a method of constructing robust coarse spaces via generalized eigenproblems in the overlaps in the two-level additive Schwarz setting. In this subsection, we will compare the MS-GFEM preconditioner with the GenEO preconditioner. 



The GenEO coarse space is based on local eigenproblems similar to \cref{eigenproblemdiscrete}. With the notation from \cref{modelproblem}, for each subdomain $\omega_i$, $i=1,\cdots,M$, the GenEO eigenproblem is defined by: Find $\lambda_i\in \mathbb{R}\cup \{+\infty\}$ and $\phi_{i} \in V_{h}(\omega_i)$ such that
\begin{equation} \label{eigenproblemgeneo}
   a_{\omega_i^{o}}\big(\chi_{h,i} (\phi_{i}), \,\chi_{h,i} (v) \big)=\lambda_{i} \,a_{\omega_i}(\phi_{i}, v) \quad \tforall v \in V_{h}(\omega_i),
\end{equation}
where $\omega_i^{o}$ denotes the overlapping zone of subdomain $\omega_i$, i.e.,
 $$
    \omega_i^{o} := \{ x \in \omega_i \ | \ x \in \omega_j \text{ for some } j \neq i \}.
 $$
As in MS-GFEM, the GenEO coarse space $V_{H}$ is built by gluing selected local eigenfunctions together with the partition of unity. Let ${\bf R}_{H}^{T}$ and ${\bf R}_{i}^{T}$ ($1\leq i\leq M$) be the matrix representations of the embedding $V_{H}\rightarrow V_{h}$ and the extensions $V_{h,0}(\omega_i)\rightarrow V_{h}$, respectively. Then, the standard two-level additive Schwarz preconditioner with the GenEO coarse space is defined by
\begin{equation} \label{preconditionergeneo}
{\bf M}^{-1}_{AS,2} = \sum_{i=1}^{M} {\bf R}_i^T ({\bf R}_{i}{\bf A}{\bf R}_{i}^{T})^{-1} {\bf R}_i
    + {\bf R}_{H}^T ({\bf R}_{H}{\bf A}{\bf R}_{H}^{T})^{-1} {\bf R}_{H}.
\end{equation}
It was shown \cite[Theorem 3.22]{spillane2014abstract} that the condition number of the preconditioner ${\bf M}^{-1}_{AS,2}$ can be bounded by 
\begin{equation}\label{conditionnumber-geneo}
        \kappa({\bf M}^{-1}_{AS,2}{\bf A}) \leq (1 + \xi) \left( 2 + \xi (2\xi + 1) \right)\max_{1\leq i\leq M} \left( 1 + \lambda_{i,m_i+1}\right),
\end{equation}
where $\xi$ is defined by \cref{coloring-constant}, and $\lambda_{i,m_i +1}$ represents the eigenvalue corresponding to the first eigenfunction on subdomain $\omega_i$ that was not included in the GenEO coarse space. Since ${\bf M}^{-1}_{AS,2}$ is symmetric, it can be used with the Conjugate Gradient (CG) method with the convergence rate 
\begin{equation*}
{\Big(\sqrt{\kappa({\bf M}^{-1}_{AS,2}{\bf A})}-1\Big)}\Big /{\Big(\sqrt{\kappa({\bf M}^{-1}_{AS,2}{\bf A})}+1\Big)}.   
\end{equation*}

Based on the description above, we see that there are several important differences between the two preconditioners. First, the local eigenproblems, while similar in form, are essentially different. The GenEO eigenproblems use no oversampling, and are posed on usual finite element spaces instead of $a$-harmonic subspaces. Consequently, the corresponding eigenvalues do not decay rapidly to $0$ as in MS-GFEM. Indeed, numerical experiments have shown that the spectra of the GenEO eigenproblems typically have an accumulation point around $1$. On the other hand, it requires more work to solve the MS-GFEM eigenproblems than the GenEO eigenproblems. Second, the MS-GFEM preconditioner is a RAS type preconditioner with the local solves performed on the oversampling domains. Moreover, the coarse space is added to the one-level method multiplicatively. This allows us to make full use of the MS-GFEM approximation theory. In contrast, the GenEO preconditioner is a standard, fully additive two-level Schwarz preconditioner. Finally, thanks to the exponential convergence property of MS-GFEM, the GMRES method preconditioned with MS-GFEM converges much faster than the CG (and also GMRES) method preconditioned with GenEO (indeed, it can be made 'arbitrarily' fast by enriching the coarse space). To summarize, the MS-GFEM preconditioner yields a much faster convergence with a more expensive setup. Whether it leads to reduced total computational cost depends on the problem to be solved -- in general, the more the preconditioner is reused, the more computational savings there will be.

We conclude this subsection by noting that conversely, we can use the MS-GFEM coarse space in a standard additive Schwarz method, or the GenEO coarse space in a restricted hybrid method. Numerical results in \cref{experimentSection} will show that for both preconditioners, the restricted hybrid method has a much faster convergence speed than the standard additive one.

\section{Abstract theory}\label{abstractFramework}
In this section, we generalize the results in \cref{iterativesection} to abstract, symmetric positive definite problems. To do this, we first present the abstract theory of MS-GFEM developed in \cite{ma2023unified}.

\subsection{Abstract MS-GFEM}
Whereas the abstract MS-GFEM was designed in both the finite and infinite dimensional settings, we restrict ourselves to the finite-dimensional case. Let $V_{0}(\Omega)$ be a finite-dimensional space of functions defined on $\Omega$, and let $a(\cdot,\cdot)$ be a symmetric, positive definite bilinear form on $V_{0}(\Omega)$. Given an element $F\in V_{0}(\Omega)^{\prime}$, we consider the problem of finding $u\in V_{0}(\Omega)$ such that
\begin{equation}\label{eq:abstract_problem}
    a(u, v) = F(v)\quad \forall v\in V_{0}(\Omega).
\end{equation}

To formulate the abstract MS-GFEM, we first introduce a family of local function spaces, local bilinear forms, and related operators. These notions are often standard in the definition of domain decomposition methods.
\begin{assumption}\label{abstract_assumption}
There exist function spaces $\big\{V(D),\;V_{0}(D): D\subset\Omega \big\}$ such that
\vspace{0.5ex}
\begin{itemize}
\item[(i)] $(\mathrm{Continuous \;\,inclusion}).$ For any subdomain $D\subset \Omega$, $V_{0}(D)\subset V(D)$. Moreover, $V(\Omega)=V_{0}(\Omega)$.

\vspace{0.5ex}
\item[(ii)] $(\mathrm{Zero\;\, extension}).$ For any subdomains $D\subset D^{\ast}$, there exists a linear operator $E_{D,D^{\ast}}:V_{0}(D)\rightarrow V_{0}(D^{\ast})$ such that 
\begin{equation*}
\big\Vert E_{D,D^{\ast}}(u) \big\Vert_{V_{0}(D^{\ast})} =  \Vert u\Vert_{V_{0}(D)}\quad \forall u\in  V_{0}(D).
\end{equation*}
In particular, $\Vert u\Vert_{V_{0}(D)} = \big\Vert E_{D,\Omega}(u) \big\Vert_{V_{0}(\Omega)}$.

\vspace{0.5ex}
\item[(iii)] $(\mathrm{Restriction}).$ For any subdomains $D\subset D^{\ast}$, there exists an operator $R_{D^{\ast},D}:V(D^{\ast})\rightarrow V(D)$ satisfying $R_{D^{\ast},D}\big(E_{D,D^{\ast}}(u)\big)=u$ for $u\in V_{0}(D)$. Moreover,
\begin{equation*}
\begin{array}{lll}
{\displaystyle R_{D^{\ast},D}\circ R_{D^{\ast\ast},D^{\ast}} = R_{D^{\ast\ast},D}\quad \;\;\;\forall D\subset D^{\ast}\subset D^{\ast\ast},}\\[2mm]
{\displaystyle \big\Vert R_{D^{\ast},D}(u) \big\Vert_{V(D)}\leq \Vert u\Vert_{V(D^{\ast})} \quad \forall u\in V(D^{\ast}). }
\end{array}
\end{equation*}
If no ambiguity arises, we simply write $R_{D}$ and denote $R_{D}(u)$ by $u|_{D}$.

\vspace{0.5ex}
\item[(iv)] $(\mathrm{Local\;\, bilinear \;\,forms}).$ For any subdomain $D$, there is a bounded, symmetric positive semi-definite bilinear form $a_{D}(\cdot,\cdot)$ on $V(D)$ with $a_{\Omega}(\cdot,\cdot) = a(\cdot,\cdot)$. In addition, if $D\subset D^{\ast}$, then for any $u\in V(D^{\ast})$, $v\in V_{0}(D)$, 
\begin{equation*}
a_{D}(u|_{D},\,v) = a_{D^{\ast}}\big(u, \,E_{D,D^{\ast}}(v)\big).
\end{equation*}
\end{itemize}
\end{assumption}
For ease of notation, we simply identify $u\in V_{0}(D)$ with its zero extension $E_{D,D^{\ast}}(u)$. Moreover, for any $D\subset \Omega$, we define 
\begin{equation*}
\Vert u\Vert_{a,D} := \sqrt{a(u,u)}  \quad \text{for all}\;\, u\in V(D).
\end{equation*}
If $D=\Omega$, we simply write $\Vert u\Vert_{a}$. Note that the local bilinear form $a_{D}(\cdot,\cdot)$ is positive definite on $V_{0}(D)$.

Let $\{ \omega_{i} \}_{i=1}^{M}$ be a collection of overlapping subdomains of $\Omega$ such that $\cup_{i=1}^{M} \omega_{i} = \Omega$. We now introduce an abstract partition of unity as follows.
\begin{definition}\label{def:2-1-1}
Let $\chi_{i}: V(\omega_{i})\rightarrow V_{0}(\omega_{i})$, $i=1,\ldots,M$, be a set of bounded linear operators such that 
\begin{equation*}
u=\sum_{i=1}^{M}\chi_{i}(u|_{\omega_{i}}) \qquad {\rm for \;\,all} \;\,u\in V_{0}(\Omega).
\end{equation*}
Then $\{\chi_{i}\}_{i=1}^{M}$ is called an abstract partition of unity subordinate to $\{ \omega_{i} \}_{i=1}^{M}$. 
\end{definition}

For each subdomain $\omega_{i}$, $i=1,\cdots,M$, let a local particular function $u_{i}^{p}\in V(\omega_{i})$ and a local approximation space $S_{m_i}(\omega_{i}) \subset V(\omega_{i})$ of dimension $m_{i}$ be given. As in the classical GFEM, the global particular function and the global approximation space are defined by gluing the local components together using the partition of unity:
\begin{equation*}
u^{p}=\sum_{i=1}^{M}\chi_{i}(u_{i}^{p}),\quad S_{m}(\Omega) =\Big\{\sum_{i=1}^{M}\chi_{i}(\phi_{i})\,:\, \phi_{i}\in S_{m_{i}}(\omega_{i})\Big\}.
\end{equation*}
The GFEM approximation of problem \cref{eq:abstract_problem} is then defined similarly as before by:
\begin{equation}\label{abstract_GFEM_solution}
{\rm Find}\;\;u^{G}\in u^{p}+ S_{m}(\Omega)\quad \;{\rm such \;\;that}\;\quad a(u^{G},v) = F(v)\quad \forall v\in S_{m}(\Omega).
\end{equation}

Before stating the fundamental approximation theorem for the abstract GFEM, we introduce the coloring constant associated with the open cover $\{\omega_{i}\}_{i=1}^{M}$, which is typically equal to the maximal number of $\omega_{i}$'s overlapping at any one point.
\begin{definition}\label{def:coloring_constant}
Let $\zeta$ be the smallest positive integer such that the set of the subdomains $\{\omega_i\}_{i=1}^{M}$ can be partitioned into $\zeta$ classes $\{\mathcal{C}_{j}, \;1\leq j\leq \zeta\}$ that satisfy
\begin{equation*}
\{\omega_{j_{1}},\,\omega_{j_{2}}\}\subset \mathcal{C}_{j} \;\;\text{for some}\;\;\mathcal{C}_{j}\; \Longleftrightarrow\; a(u_{1},u_{2}) = 0\;\;\; \forall u_{1}\in V_{0}(\omega_{j_{1}}),\;u_{2}\in V_{0}(\omega_{j_{2}}).
\end{equation*}
Then $\zeta$ is called the coloring constant associated with the open cover $\{\omega_{i}\}_{i=1}^{M}$.
\end{definition}

\begin{theorem}\label{thm:2-1}
Let $u\in V_{0}(\Omega)$ be the solution of \cref{eq:abstract_problem}, and $u^{G}$ be the GFEM approximation defined by \cref{abstract_GFEM_solution}. Assuming that for each $i=1,\cdots,M$, 
\begin{equation*}
\inf_{v_{i}\in u_{i}^{p}+ S_{m_{i}}(\omega_{i})} \big\Vert \chi_{i}(u|_{\omega_i} - v_{i})\big\Vert_{a,\omega_i} \leq \varepsilon_{i},
\end{equation*}
then
\begin{equation*}
\Vert u - u^{G}\Vert_{a} = \inf_{v\in u^{p}+ S_{m}(\Omega)}\Vert u - v\Vert_{a} \leq \Big(\zeta \sum_{i=1}^{M}\varepsilon_{i}^{2}\Big)^{1/2},
\end{equation*}
where $\zeta$ is the coloring constant defined in \cref{def:coloring_constant}.
\end{theorem}

Next we construct the local particular functions and local approximation spaces for the abstract GFEM. For each subdomain $\omega_i$, let $\omega_i^{\ast}$ be the associated oversampling domains with ${\rm dist}(\omega_i,\partial\omega_i^{\ast}\setminus\partial \Omega) > 0$. We consider the following local problems:
\begin{equation}
  \text{Find}\;\,\psi_{i}\in V_{0}(\omega^{\ast}_i) \quad \text{such that}\quad  a_{\omega_i^{\ast}}(\psi_i, v) =  F(v)\quad \forall v\in V_{0}(\omega^{\ast}_i).
\end{equation}
The local particular function on $\omega_i$ is then defined by 
\begin{equation}\label{abstract_loc_par_sol}
u_{i}^{p} = \psi_{i}|_{\omega_i}. 
\end{equation}
To construct the local approximation spaces, we define the abstract $a$-harmonic spaces
\begin{equation}\label{abstract_harmonic_space}
V_{a}(\omega_{i}^{\ast}) = \big\{u\in V(\omega_{i}^{\ast})\,:\, a_{\omega_i^{\ast}}(u,v) = 0\quad \forall v\in V_{0}(\omega_{i}^{\ast}) \big\},
\end{equation}
and consider the following local eigenproblems: Find $\lambda_{i}\in \mathbb{R}\cup \{+\infty\}$ and $\phi_{i}\in V_{a}(\omega_{i}^{\ast})$ such that
\begin{equation}\label{abstract_eigenproblem}
a_{\omega_i}\big(\chi_{i}(\phi_{i}|_{\omega_i}),\,  \chi_{i} (v|_{\omega_i})\big) = \lambda_{i}\,a_{\omega_{i}^{\ast}}(\phi_{i},  v)\quad \forall v\in V_{a}(\omega_{i}^{\ast}).
\end{equation}
The desired local approximation space $S_{n_i}(\omega_i)$ is then defined by
\begin{equation}\label{abstract_local_space}
S_{m_{i}}(\omega_i) =  {\rm span}\big\{{\phi_{i,1}}|_{\omega_i},\ldots,{\phi_{i,m_{i}}}|_{\omega_i}\big\},
\end{equation}
where $\phi_{i,j}$ denotes the $j$-th eigenfunction (with eigenvalues arranged in decreasing order) of problem \cref{abstract_eigenproblem}.




With $u_{i}^{p}$ and $S_{m_i}(\omega_i)$ constructed above, we have the following local approximation error estimates analogous to \cref{local-appro-property} provided that $m_{i}\geq l_i$, where $l_i$ denotes the dimension of the kernel of $a_{\omega_i^{\ast}}(\cdot,\cdot)$ on $V_{a}(\omega_{i}^{\ast})$.

\begin{lemma}\label{abstract-local-appro-proerty}
    Let $u$ be the solution of problem \cref{eq:abstract_problem}, and let $u_{i}^{p}$ and $S_{m_i}(\omega_i)$ $(m_i\geq l_{i})$ be defined by \cref{abstract_loc_par_sol} and \cref{abstract_local_space}, respectively. Then,
    $$
        \inf\limits_{\phi \in u_{i}^{p} + S_{m_i}(\omega_i)} \lVert \chi_{i} (u|_{\omega_i} - \phi) \rVert_{a,{\omega_i}} \leq \lambda_{i,m_i+1}^{1/2} \lVert u \rVert_{a,\omega_i^*},
    $$
where $\lambda_{i,m_i+1}$ denotes the $(m_i+1)$-th eigenvalue of problem \cref{abstract_eigenproblem}.   
\end{lemma}

Finally, by combining \cref{thm:2-1} and \cref{abstract-local-appro-proerty}, we can get the same error estimate for the abstract MS-GFEM as \cref{msgfemapprox}. In particular, we define the error bound $\Lambda$ similarly to \cref{Lamdba} by 
    \begin{equation}\label{abstract_Lamdba}
     \Lambda := \Big( \xi \xi^{\ast}\max\limits_{i=1,..,M} \lambda_{i,m_i+1} \Big)^{1/2},    
    \end{equation}
where $\zeta$ and $\zeta^{\ast}$ denotes the coloring constants associated with the subdomains $\{\omega_{i}\}_{i=1}^{M}$ and the oversamling domains $\{\omega^{\ast}_{i}\}_{i=1}^{M}$, respectively.

\subsection{Exponential convergence of abstract MS-GFEM} \label{exponentialdecaysection}
The core of the abstract MS-GFEM theory is the local exponential convergence under two fundamental conditions -- a Caccioppoli-type inequality and a weak approximation property. To state these two conditions, we first assume that there exists a family of local Hilbert spaces $\{\mathcal{L}(D)\}_{D\subset\Omega}$ satisfying that (i) for any $D\subset \Omega$, $V(D)\subset \mathcal{L}(D)$, and (ii) for any $D\subset D^{\ast}$ and $u\in \mathcal{L}(D^{\ast})$, $u|_{D}\in \mathcal{L}(D)$ with $ \Vert u|_{D}\Vert_{\mathcal{L}(D)}\leq \Vert u\Vert_{\mathcal{L}(D^{\ast})}$. 

\begin{assumption}[Caccioppoli-type inequality]\label{Caccioppoli inequality}
There exists a constant $C_{\rm cac}>0$ such that for any subdomains $D\subset D^{\ast}$ with $\delta := {\rm dist} ({D},\partial D^{\ast}\setminus \partial \Omega)>0$, 
\begin{equation*}
\Vert u\Vert_{a,D} \leq ({C_{\rm cac}}/{\delta}) \Vert u\Vert_{\mathcal{L}({D}^{\ast}\setminus {D})} \quad \text{for all}\;\; u\in V_{a}({D}^{\ast}).
\end{equation*} 
\end{assumption}

\begin{assumption}[Weak approximation property]\label{weak-approximation}
Let $D\subset D^{\ast}\subset {D}^{\ast\ast}$ be subdomains of $\Omega$ such that $\hat{\delta}:={\rm dist} ({D}^{\ast},\partial D^{\ast\ast}\setminus \partial \Omega)>0$ and that ${\rm dist} ({\bm x},\partial D^{\ast}\setminus \partial \Omega)\leq \hat{\delta}$ for all ${\bm x}\in \partial D\setminus \partial \Omega$. For each $m\in \mathbb{N}$, there exists an $m$-dimensional space $Q_{m}(D^{\ast\ast})\subset \mathcal{L}({D}^{\ast\ast})$ such that for all $u\in V_{a}(D^{\ast\ast})$,
 \begin{equation}
\inf_{v\in Q_{m}(D^{\ast\ast})} \big\Vert u-v\big\Vert_{\mathcal{L}({D}^{\ast}\setminus D)} \leq C_{\rm wa}\big|\mathsf{V}_{\hat{\delta}}({D}^{\ast}\setminus D)\big|^{\alpha} m^{-\alpha}\Vert u\Vert_{a, D^{\ast}},
 \end{equation}
 where $\big|\mathsf{V}_{\hat{\delta}}({D}^{\ast}\setminus D)\big|: = {\rm vol}\big(\mathsf{V}_{\hat{\delta}}({D}^{\ast}\setminus D)\big)$, $\mathsf{V}_{\hat{\delta}}({D}^{\ast}\setminus D):= \big\{{\bm x}\in {D}^{\ast\ast}: {\rm dist}({\bm x}, {D}^{\ast}\setminus D)\leq \hat{\delta} \big\}$, and $C_{\rm wa}$ and $\alpha$ are positive constants independent of $D^{\ast\ast}$, $D^{\ast}$, $D$, and $m$.
\end{assumption}

Now we can give the central result of the abstract MS-GFEM theory, which states that the eigenvalues of problems \cref{abstract_eigenproblem} (and thus the local approximation errors) decay exponentially under the two fundamental conditions.
\begin{theorem}\cite[Theorem 3.7]{ma2023unified}
Let Assumptions~\ref{Caccioppoli inequality} and \ref{weak-approximation} be satisfied, and let $\lambda_{i,k}$ be the $k$-th eigenvalue of problem \cref{abstract_eigenproblem}. Then, there exist positive constants $k_{i}$, $b_{i}$, and $C_{i}$, such that for any $k>k_{i}$,
\begin{equation}
\lambda_{i,k}^{1/2}\leq C_{i} \,e^{-b_{i}k^{\alpha}},   
\end{equation}
where $C_{i}$ only depends on the partition of unity operator $\chi_i$, and $\alpha>0$ is the same constant as in Assumption~\ref{weak-approximation}.     
\end{theorem}

\subsection{Abstract two-level RAS methods}
Based on the abstract MS-GFEM, we can define the corresponding iterative method in a fully abstract way following \cref{iterativesection}. No additional assumptions or conditions are needed in this subsection. 

Let $\pi_{S}$ and $\pi_{i}$ ($1\leq i\leq M$) denote the $a$-orthogonal projections of $V_{0}(\Omega)$ onto the coarse space $S_{m}(\Omega)$ and the local spaces $V_{0}(\omega_i^{\ast})$, respectively. Moreover, for each $i=1,\cdots,M$, we define the operator $\widetilde{\chi}_{i}: V_{0}(\omega_i^*)\rightarrow V_{0}(\omega_i^*)$ by $\widetilde{\chi}_{i}(v) = {\chi}_{i}(v|_{\omega_i})$. Then, we can define the map $G:V_{0}(\Omega)\rightarrow V_{0}(\Omega)$ of the abstract MS-GFEM approximation as in \cref{MSGFEM-map} by
\begin{equation}\label{abstract_MSGFEM_map}
G(v) = \sum\limits_{i=1}^M \widetilde{\chi}_{i} \pi_i(v) + \pi_S \Big( v - \sum\limits_{i=1}^M \widetilde{\chi}_{i}\pi_i(v) \Big).
\end{equation}
With this map, the abstract analogue of the iterative MS-GFEM is defined identically as in \cref{iterativesection}. 
\begin{definition}[abstract iterative MS-GFEM]
Let $u$ be the solution of \cref{eq:abstract_problem}. For an initial guess $\richIterate^0 \in V_{0}(\Omega)$, let the sequence $\{\richIterate^j\}_{j \in \mathbb{N}}$ be generated by 
    \begin{equation}\label{abstract_MSGFEM_iteration} 
        \richIterate^{j+1} = \richIterate^j + G\left( u - \richIterate^j \right),\quad j=0,1,\cdots.
    \end{equation}
\end{definition}
Similar to \cref{cor2}, we have the following convergence estimate. 
\begin{proposition} \label{rate_abstract_Richardson}
Let $\Lambda$ be the defined by \cref{abstract_Lamdba}, and let the sequence $\{\richIterate^j\}_{j \in \mathbb{N}}$ be generated by \cref{abstract_MSGFEM_iteration}. Then, for $j=1,2,\cdots$,
    $$
        \lVert \richIterate^{j+1} - u \rVert_a \leq \rate \lVert \richIterate^j - u \rVert_{a}.
    $$  
\end{proposition}

Next we follow \cref{iterativesection} to give the matrix formulation of the iterative method above. Given a basis for $V_{0}(\Omega)$, problem \cref{eq:abstract_problem} can be written as a linear system
\begin{equation}\label{abstract_linear_system}
    {\bf A} {\bf u} = {\bf f},
\end{equation}
where ${\bf A}\in \mathbb{R}^{n\times n}$ and ${\bf f}\in \mathbb{R}^{n}$ with $n:={\rm dim}(V_{0}(\Omega))$. Let ${\bf R}_S^T$, $\widetilde{\bf R}_i^T$, and $\widetilde{\boldsymbol{\chi}}_i$ ($1\leq i\leq M$) be the matrix representations of the embedding $S_{m}(\Omega) \rightarrow V_{0}(\Omega)$, the inclusions $V_{0}(\omega_i^*) \rightarrow V_{0}(\Omega)$, and the operators $\widetilde{\chi}_{i}$, respectively. Denoting by
$$
    \widetilde{\bf A}_i: = \widetilde{\bf R}_i {\bf A} \widetilde{\bf R}_i^T \quad \text{and} \quad {\bf A}_S: = {\bf R}_S {\bf A} {\bf R}_S^T,
$$
then the map $G$ \cref{abstract_MSGFEM_map} has the following matrix representation:
\begin{align*}
  {\bf G} 
    &= \left(  \sum_{i=1}^{M} \widetilde{\bf R}_i^T \widetilde{\boldsymbol{\chi}}_i \widetilde{\bf A}_i^{-1} \widetilde{\bf R}_i {\bf A} \right)
    + 
    \big({\bf R}_S^T {\bf A}_S^{-1} {\bf R}_S {\bf A} \big) 
    \left( {\bf I} - \sum_{i=1}^{M} \widetilde{\bf R}_i^T \widetilde{\boldsymbol{\chi}}_i \widetilde{\bf A}_i^{-1} \widetilde{\bf R}_i {\bf A}  \right),
\end{align*}
where ${\bf I}$ denotes the identity matrix. With this representation, the abstract MS-GFEM iteration \cref{abstract_MSGFEM_iteration} can be written in matrix form as
\begin{equation} \label{abstract-matrix-Richardson-iteration}
    \richIterateVector^{j+1} = \richIterateVector^{j} + {\bf G}({\bf u} - \richIterateVector^{j}).
\end{equation}
Equation \cref{abstract-matrix-Richardson-iteration} is a preconditioned Richardson iteration for the linear system \cref{abstract_linear_system}, with the preconditioner given by
\begin{equation} \label{abstract_preconditioner}
 {\bf B} = \left(  \sum_{i=1}^{M} \widetilde{\bf R}_i^T \widetilde{\boldsymbol{\chi}}_i \widetilde{\bf A}_i^{-1} \widetilde{\bf R}_i \right)
    + 
    \big({\bf R}_S^T {\bf A}_S^{-1} {\bf R}_S \big) 
    \left( {\bf I} - {\bf A}\sum_{i=1}^{M} \widetilde{\bf R}_i^T \widetilde{\boldsymbol{\chi}}_i \widetilde{\bf A}_i^{-1} \widetilde{\bf R}_i\right).
\end{equation}

Finally, we consider solving the preconditioned system ${\bf BAu} = {\bf Bf}$ by GMRES. We assume that GMRES is applied with respect to the inner product $b(\cdot,\cdot)$ on $\mathbb{R}^{n}$ which satisfies that for all ${\bf v}\in \mathbb{R}^{n}$,
\begin{equation*}
 b_{1} \sqrt{b({\bf v}, {\bf v})}\leq \sqrt{{\bf v}^{T}{\bf A}{\bf v}} \leq b_{2} \sqrt{b({\bf v}, {\bf v})}
\end{equation*}
for some $b_{1},b_{2}>0$. Let $\{{\bf u}^{j} \}_{j\in\mathbb{N}} $ be the approximation sequence generated by GMRES starting from an initial vector ${\bf u}^{0}$. We have the following convergence result.
\begin{theorem}\label{abstract_gmresconvergence}
If $\Lambda<1$, after $j$ steps, the norm of the residual of the GMRES algorithm can be bounded by 
    $$
        \lVert {\bf BA}({\bf u} - {\bf u}^{j}) \rVert_b \leq \rate^j \left( \frac{1+\rate}{1-\rate}\frac{b_2}{b_1}  \right) \lVert {\bf BA}({\bf u} - {\bf u}^{0}) \rVert_b.
    $$
\end{theorem}
The proof is exactly the same as that of \cref{gmresconvergence}.

\subsection{Applications}
We have established the two-level RAS methods with convergence analysis based on the abstract theory of MS-GFEM. Our theory formally relies only on the general assumption -- \cref{abstract_assumption}, which can be verified easily for various elliptic problems. The core of the theory, the exponential convergence property of local approximations, however, requires two nontrivial conditions -- the Caccioppoli-type inequality and the weak approximation property. In this subsection, we present two applications of the abstract theory, with all the assumptions or conditions verified, to demonstrate its usefulness. We should note that the application of our theory is not limited to these two examples. Indeed, it can be proved that the ${\bm H}({\rm curl})$ elliptic problems with N\'{e}d\'{e}lec elements based discretizations also fit into this theory (the two fundamental conditions in the continuous setting were verified in \cite{ma2023unified}, and their discrete counterparts will be verified in a forthcoming paper). Moreover, with a slight modification of the two conditions, this theory also applies to higher-order problems, e.g., biharmonic type problems. We omit these extensions to avoid overloading the paper. The interested reader is referred to \cite{ma2023unified} for details.

\begin{example}[Diffusion Equation] \label{diffusionExample}
While the diffusion equation was used as the model problem before, we reconsider it here within the unified theoretical framework established in \cite{ma2023unified}. Sharper convergence rates will be provided using the newly established theoretical results in \cite{ma2023unified}.

Let $V_{0}(\Omega)\subset H_{0}^{1}(\Omega)$ be a finite element space that consists of continuous piecewise polynomials on a shape-regular triangulation. For the bilinear form we set
    $$
        a_{D}(v,w) =  \int_{D}(A\nabla u)\cdot\nabla v \,dx, \quad D\subset \Omega,
    $$
    where the coefficient $A$ satisfies the same conditions as in \cref{continuousmodelproblem}. Clearly, \cref{abstract_assumption} is trivially satisfied with 
    \begin{equation*}
    V(D) = \{v|_{D}: v\in V_{0}(\Omega) \},\quad {\rm and} \;\;   V_{0}(D) = \{v|_{D}: v\in V_{0}(\Omega),\;\;  {\rm supp}\,(v)\subset \overline{D} \}.
    \end{equation*}
The Caccioppoli-type inequality and the weak approximation property with $\alpha = 1/d$ were verified in \cite{ma2023unified}. Based on the abstract theory, the convergence rate of the local approximations is $O(\exp(-ck^{1/d}))$, which is sharper than the one given by \cref{expcont}.
\end{example}


\begin{example}[Linear Elasticity] \label{linearElasticityExample}
Let $V_{0}(\Omega)\subset (H_{0}^{1}(\Omega))^{d}$ be a finite element space of vector fields that consists of continuous piecewise polynomials on a shape-regular triangulation. The associated bilinear form is given by 
   $$
     a_{D}({\bm v},{\bm w}) = \int_{D} \sum_{i,j}\Big(\sum_{k,l}C_{ijkl}\epsilon_{kl}({\bm v})\Big)\epsilon_{ij}({\bm w}) \,dx, \quad D\subset \Omega,
   $$
where the infinitesimal strain tensor ${\bm \epsilon} = (\epsilon_{ij})$ is given by the symmetric part of the deformation gradient, i.e.
   $$
     \epsilon_{ij}({\bm v})=\frac{1}{2}\left( \frac{\partial v_i}{\partial x_j} + \frac{\partial v_j}{\partial x_i}\right),
   $$
and ${\bm C} = (C_{ijkl})$ is the fourth-order symmetric elasticity tensor that satisfies the standard boundedness and coercivity conditions. 

As in the previous example, \cref{abstract_assumption} is also trivially satisfied in this case.  The Caccioppoli-type inequality and the weak approximation property with $\alpha = 1/d$ were also proved in \cite{ma2023unified}. Similarly as above, the abstract theory shows that the convergence rate of the local approximations in this case is $O(\exp(-ck^{1/d}))$.

\end{example}

\section{Numerical experiments}\label{experimentSection}
In this section, we present numerical examples to verify our theoretical results and demonstrate practical applicability of the proposed method. The first example is a two-dimensional heterogeneous diffusion problem with a high-contrast coefficient. We implement this example in Matlab, with the purpose of verifying the predicted convergence performance, getting a general idea of the costs, and comparing the performance of related methods. The second example is a three-dimensional elasticity problem in composite aero-structures, implemented in the Distributed and Unified Numerics Environment (DUNE) software package \cite{bastian2021dune}. The purpose of this example is to demonstrate the applicability of the MS-GFEM preconditioner to real-world problems in a high performance computing environment.

\subsection{Numerical solution of local eigenproblems}
We start by presenting a solution technique for the local eigenproblems in MS-GFEM. For simplicity, we focus on the model problem in \cref{modelproblem}, whereas extending the technique to general elliptic problems is straightforward. Following \cite{ma2022error}, by introducing a Lagrange multiplier to relax the $a$-harmonic condition in $V_{h,A}(\omega)$, we can rewrite the eigenproblem \cref{eigenproblemdiscrete} in mixed formulation: Find $\mu\in\mathbb{R}\cup\{+\infty\}$, $\phi\in V_{h}(\omega^{\ast})$, and $p\in V_{h,0}(\omega^{\ast})$ such that
\begin{equation}\label{eq:saddle-point-eig}
\begin{aligned}
a_{\omega^{\ast}}(\phi,v) + a_{\omega^{\ast}}(v,p) =& \;\mu\, a_{\omega}(\chi_{h}(\phi), \chi_{h}(v))\quad\, \forall v\in V_{h}(\omega^{\ast}),\\
a_{\omega^{\ast}}(\phi,\xi) =&\;0\quad\qquad\quad \qquad\qquad\;\; \;\forall \xi\in V_{h,0}(\omega^{\ast}).
\end{aligned}
\end{equation}
Here we omit the subdomain index and write $\mu = \lambda^{-1}$ for simplicity. Note that the eigenvectors of the mixed problem \cref{eq:saddle-point-eig} with finite eigenvalues ($\mu<\infty$) correspond to the eigenvectors of problem \cref{eigenproblemdiscrete}. Therefore, the mixed problem \cref{eq:saddle-point-eig} is equivalent to the original problem \cref{eigenproblemdiscrete} for building the local approximation space. Let $\big\{\varphi_{1},..,\varphi_{n_1}\big\}$ be a basis for $V_{h,0}(\omega^{\ast})$ and $\big\{\varphi_{1},\ldots,\varphi_{n_1}\big\}\cup \big\{\varphi_{n_1+1},\ldots,\varphi_{n_1+n_2}\big\}$ a basis for $V_{h}(\omega^{\ast})$. We can formulate problem \cref{eq:saddle-point-eig} as the following matrix eigenvalue problem: Find $\mu\in\mathbb{R}\cup\{+\infty\}$, ${\bm \phi} = ({\bm \phi}_{1}, {\bm \phi}_{2})\in\mathbb{R}^{n_{1}+n_{2}}$, and ${\bm p}\in \mathbb{R}^{n_1}$ such that
\begin{equation}\label{eq:mixed_eig_matrix}
{\left( \begin{array}{ccc}
{\bf A}_{11} & {\bf A}_{12} & {\bf A}_{11}\\
{\bf A}_{21} & {\bf A}_{22} & {\bf A}_{21}\\
{\bf A}_{11} & {\bf A}_{12} & {\bf 0}
\end{array} 
\right )}
{\left( \begin{array}{c}
{\bm \phi}_{1} \\
{\bm \phi}_{2} \\
{\bm p}
\end{array} 
\right )} = \mu
{\left( \begin{array}{ccc}
{\bf P}_{11} & {\bf P}_{12} & {\bf 0}\\
{\bf P}_{21} & {\bf P}_{22} & {\bf 0}\\
{\bf 0} & {\bf 0} & {\bf 0}
\end{array} 
\right )}
{\left( \begin{array}{c}
{\bm \phi}_{1} \\
{\bm \phi}_{2} \\
{\bm p}
\end{array} 
\right )},
\end{equation}
where ${\bf A}_{ij}=\Big(a_{\omega^{\ast}}(\varphi_{k}, \varphi_{l})\Big)_{k\in I_{i}, l\in I_{j}}$ and ${\bf P}_{ij}=\Big(a_{\omega}\big(\chi_{h}(\varphi_{k}),\,  \chi_{h}(\varphi_{l})\big)\Big)_{k\in I_{i}, l\in I_{j}}$. Here, $I_{1} = \{1,\ldots,n_{1}\}$ and $I_{2} = \{n_{1}+1,\ldots,n_{1}+n_{2}\}$ are index sets. By exploiting the special structure of the matrix pencil above, an efficient algorithm was developed in \cite{ma2022error} for solving the augmented eigenproblem \cref{eq:mixed_eig_matrix} based on simple block-elimination. See \cite{ma2023unified} for its extension to general elliptic problems. Whereas this algorithm is much cheaper in terms of memory and computational time, it is slightly less accurate than the direct solution of the eigenproblem \cref{eq:mixed_eig_matrix} (but can still attain an accuracy of order $10^{-5}$). Since in this paper, we are primarily concerned with the convergence performance of the proposed method predicted by our theory, we solve the eigenproblem \cref{eq:mixed_eig_matrix} directly without introducing additional errors.


In addition to the technique of Lagrange multiplier, there exist other methods for solving the local eigenproblems based on approximating the $a$-harmonic subspaces. These approximation techniques include random sampling \cite{chen2020random,buhr2018randomized} and special choices of boundary data to build the $a$-harmonic bases \cite{babuvska2020multiscale,lipton2022angles}. We will incorporate these techniques into our iterative method and evaluate their performance in a future work.   

\subsection{Example 1: heterogeneous diffusion equation}
In this example, we solve the diffusion equation with a heterogeneous coefficient (see \cref{skyscraper} left) on the domain $\Omega = [0,1] \times [0,1]$. The following mixed boundary conditions are used:
\begin{equation*}
\left\{
\begin{array}{ll}
    A\nabla u\cdot {\bf n}  = 1\quad \;\;\,\text{on} \;\;(0,1)\times \{1 \},  \quad u = 10 \qquad \,\text{on} \;\;\{0\}\times [0,1], \\
    A\nabla u\cdot {\bf n}  = -1\quad \text{on} \;\;(0,1)\times \{0 \},  \quad u = -10 \quad \;\text{on} \;\;\{1\}\times [0,1].
\end{array}
\right.
\end{equation*}
The source term $f$ is given by
$$
    f(x):= 1000 \exp((-(x_1-0.15)(x_1-0.15)-10(x_2-0.55)(x_2-0.55))).
$$
The FE mesh is built on a uniform Cartesian grid with $h=1/700$ and the lowest order quadrilateral element is used. We first divide the computational domain into $7 \times 7$ rectangular non-overlapping subdomains, and then add two layers of elements to each subdomain to create overlapping subdomains $\{\omega_i\}$. Each subdomain $\omega_i$ is extended further by adding a few layers of elements to generate the corresponding oversampling domain $\omega_i^{\ast}$. We refer to the number of these additional layers of elements as oversampling layers or 'Ovsp' for short in the following. For the coarse space we use a certain number of eigenfunctions on each subdomain, denoted by '\#Eig'. The experiment is run on a desktop PC with an AMD Ryzen 5 2600 processor. For the eigensolves we use Matlab's eigs, a wrapper for Arpack. Local solves and coarse solves are implemented with the built-in backslash, a direct solver based on a multifrontal method. We stop the iteration once a residual reduction of $10^{-10}$ is attained.


\begin{figure} \label{skyscraper}
\begin{center}
  \includegraphics[width=.40\textwidth]{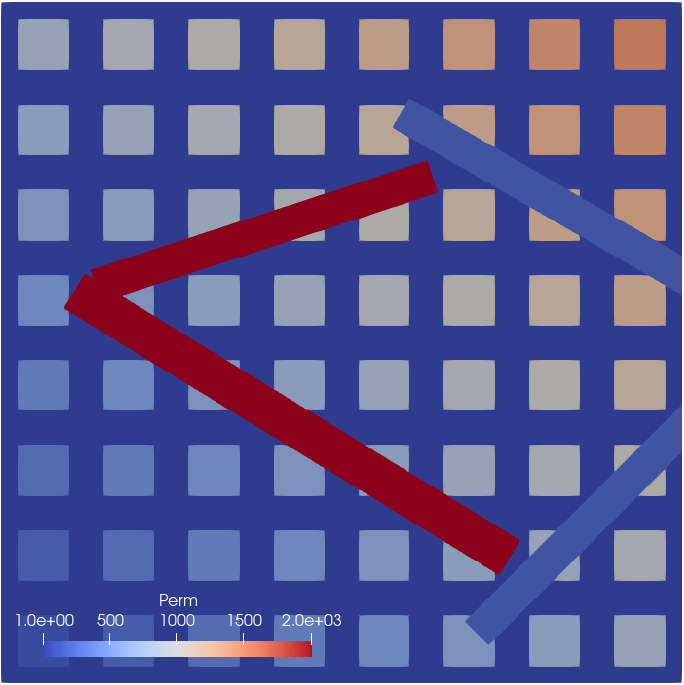}
  \hspace{3ex}
    \includegraphics[width=.40\textwidth]{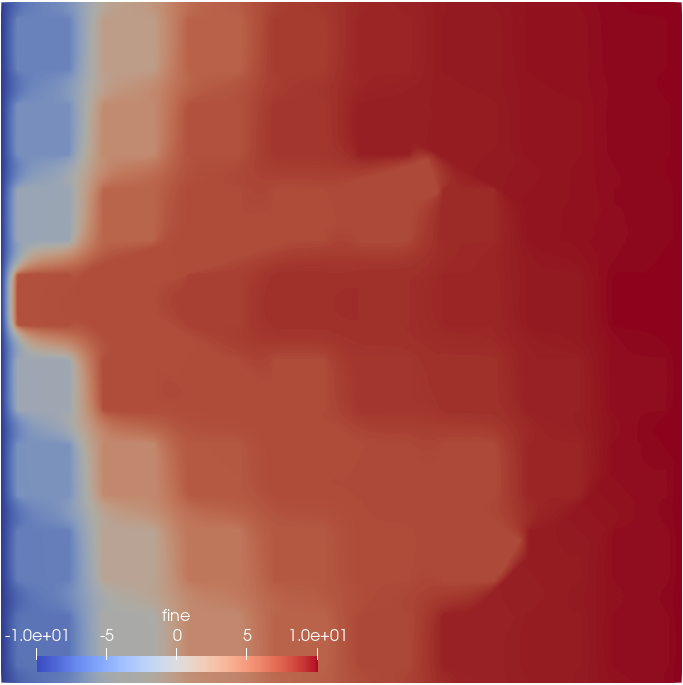}
    \caption{Diffusion equation example. Left: 'skyscraper' coefficient; right: solution.}  
\end{center}
\end{figure}

\begin{figure} \label{heatmaps}
    \includegraphics[width=0.49\textwidth]{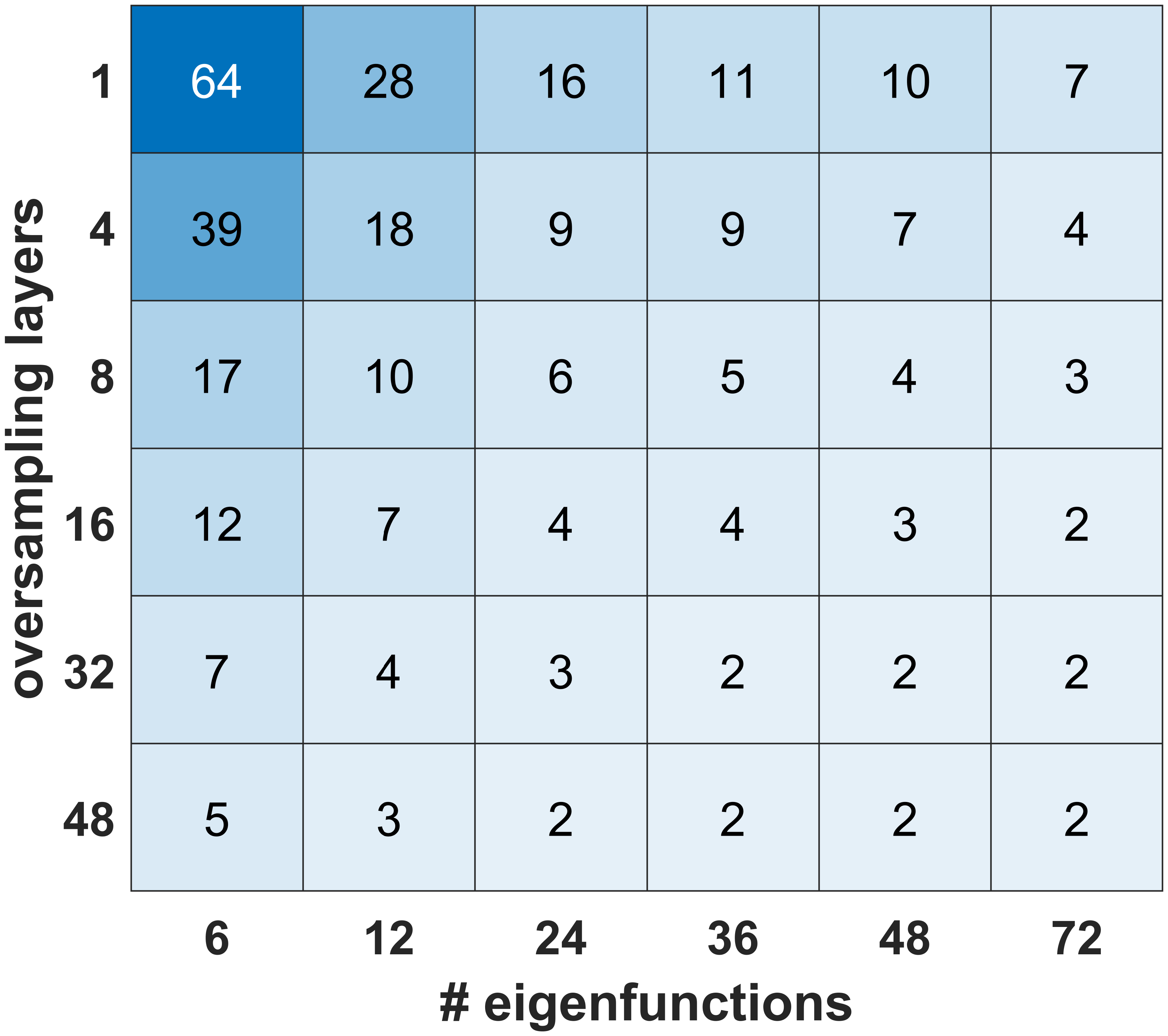}\hfill
    \includegraphics[width=0.49\textwidth]{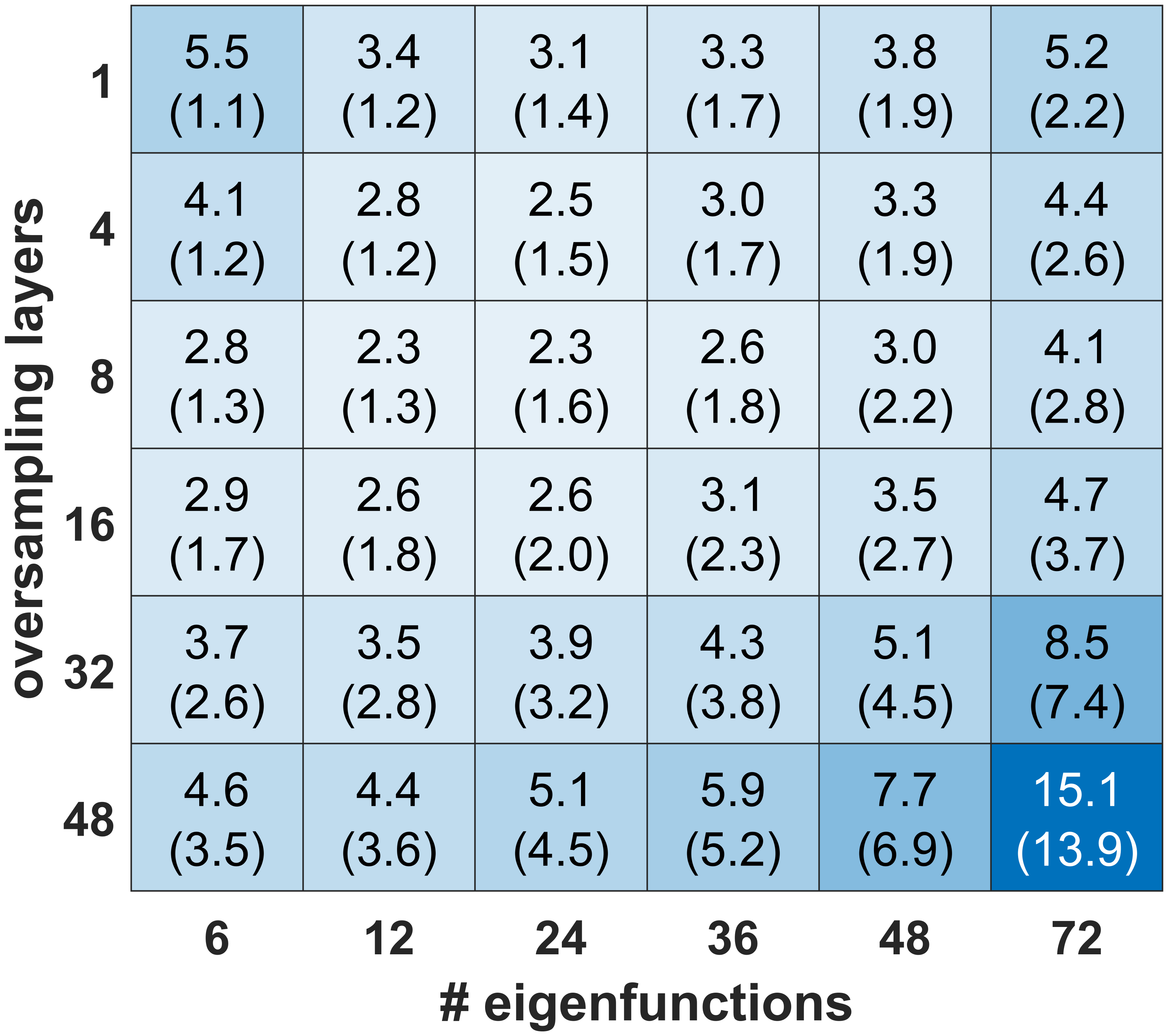}
    \caption{Results for the diffusion example solved with MS-GFEM as the iterative method. We report the iteration count (left) and total computational time in seconds (right) for different numbers of oversampling layers and local eigenfunctions used per subdomain. The numbers in brackets on the right show the time spent on the local eigensolves.}
\end{figure}


\begin{figure} \label{convergencePlots}
    \includegraphics[width=0.49\textwidth]{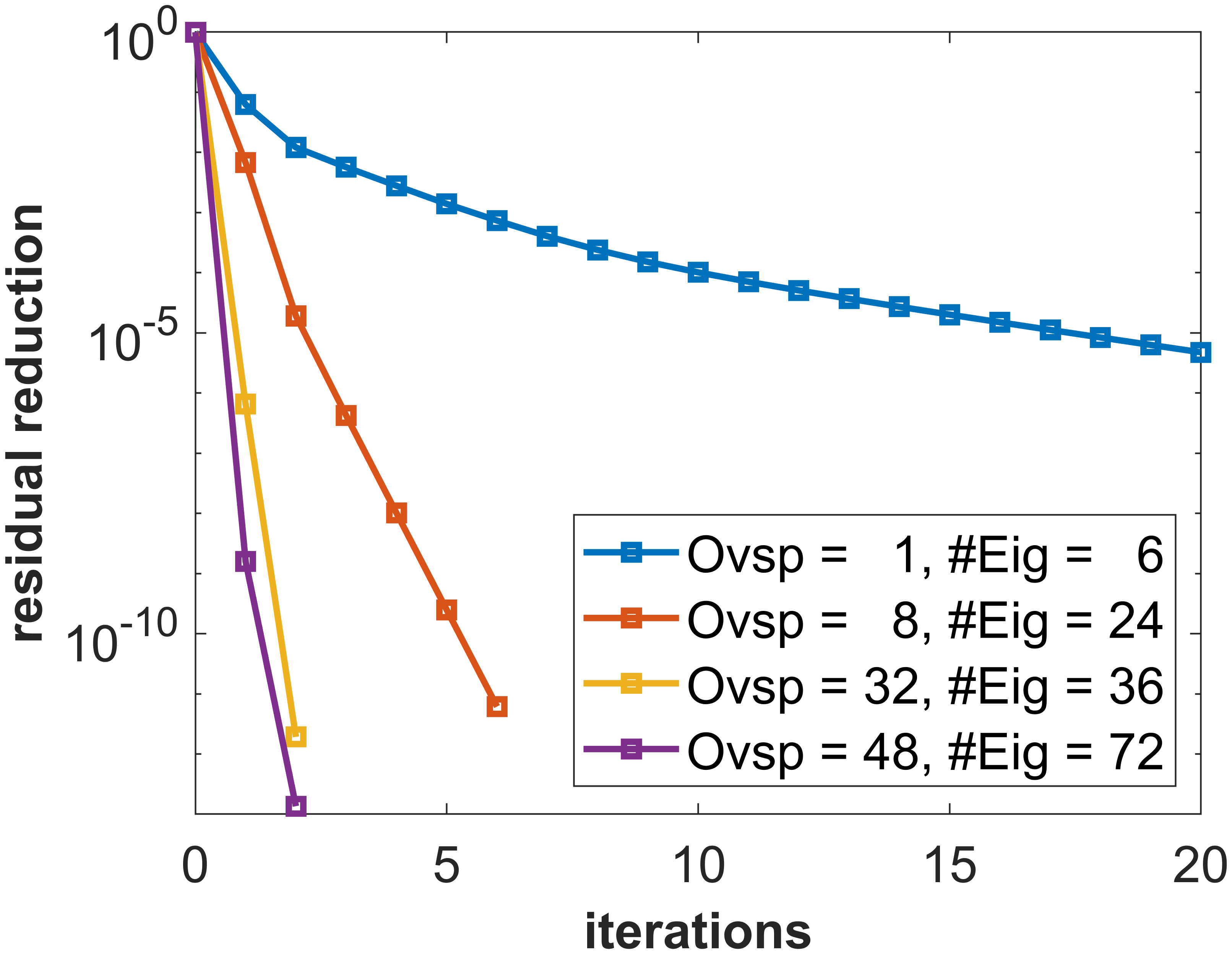}\hfill
    \includegraphics[width=0.49\textwidth]{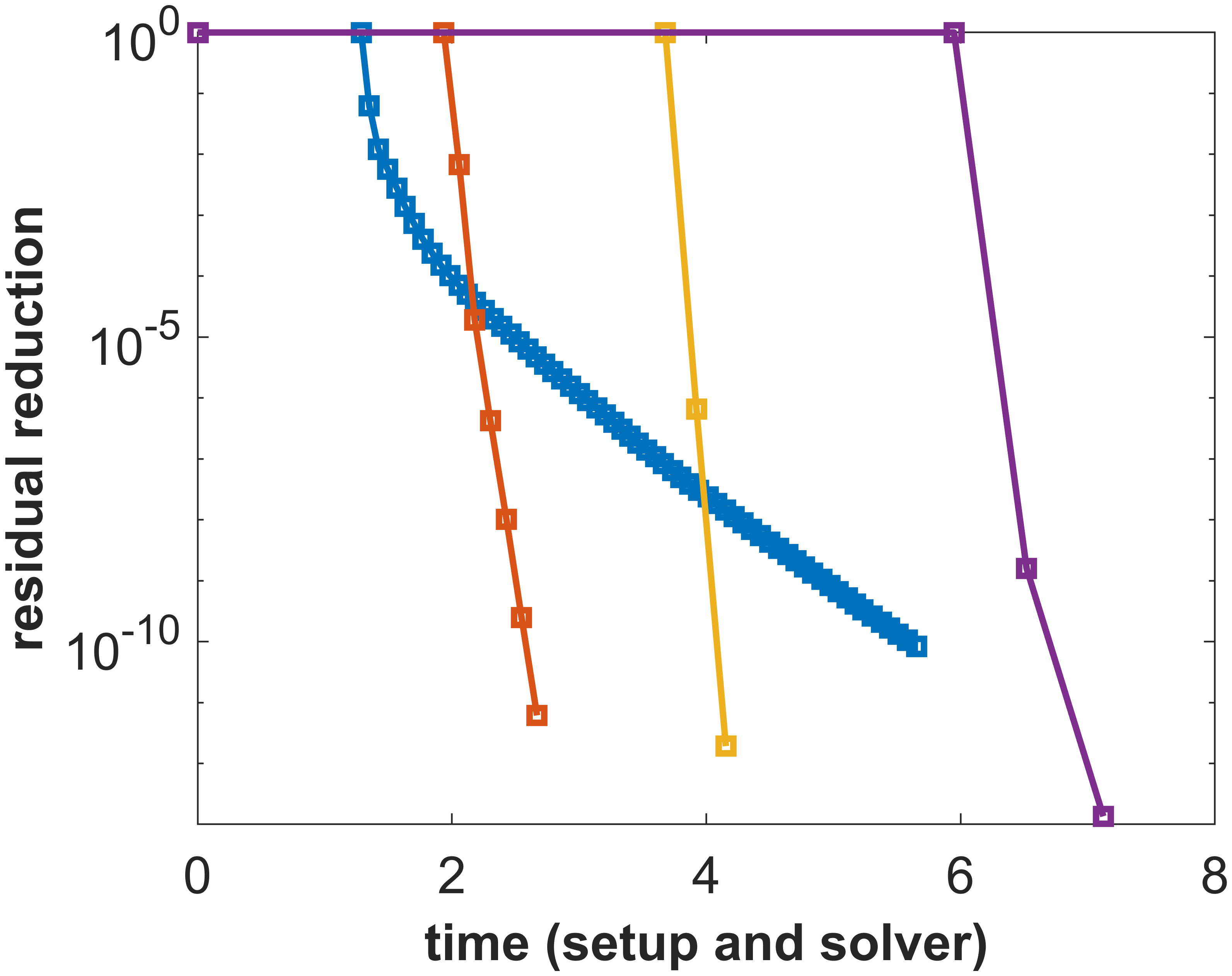}
    \caption{Results for the diffusion example solved with MS-GFEM as the iterative method. The plots show the residual reduction over iterations (left) and over time (right) for selected parameters used in \Cref{heatmaps}. }
\end{figure}


\begin{figure} \label{fig:compareSolvers}
    \includegraphics[width=.48\textwidth]{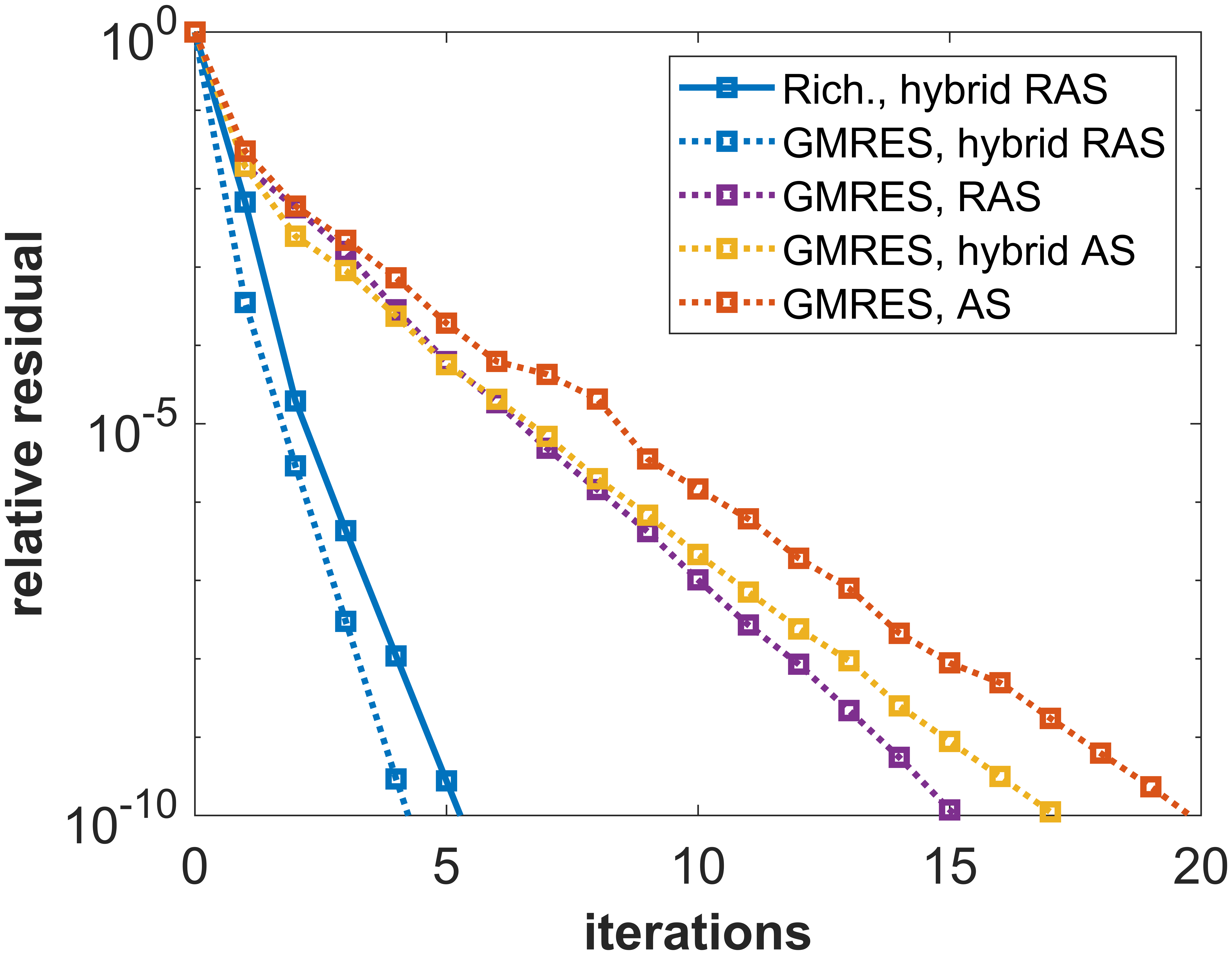}\hfill
    \includegraphics[width=.48\textwidth]{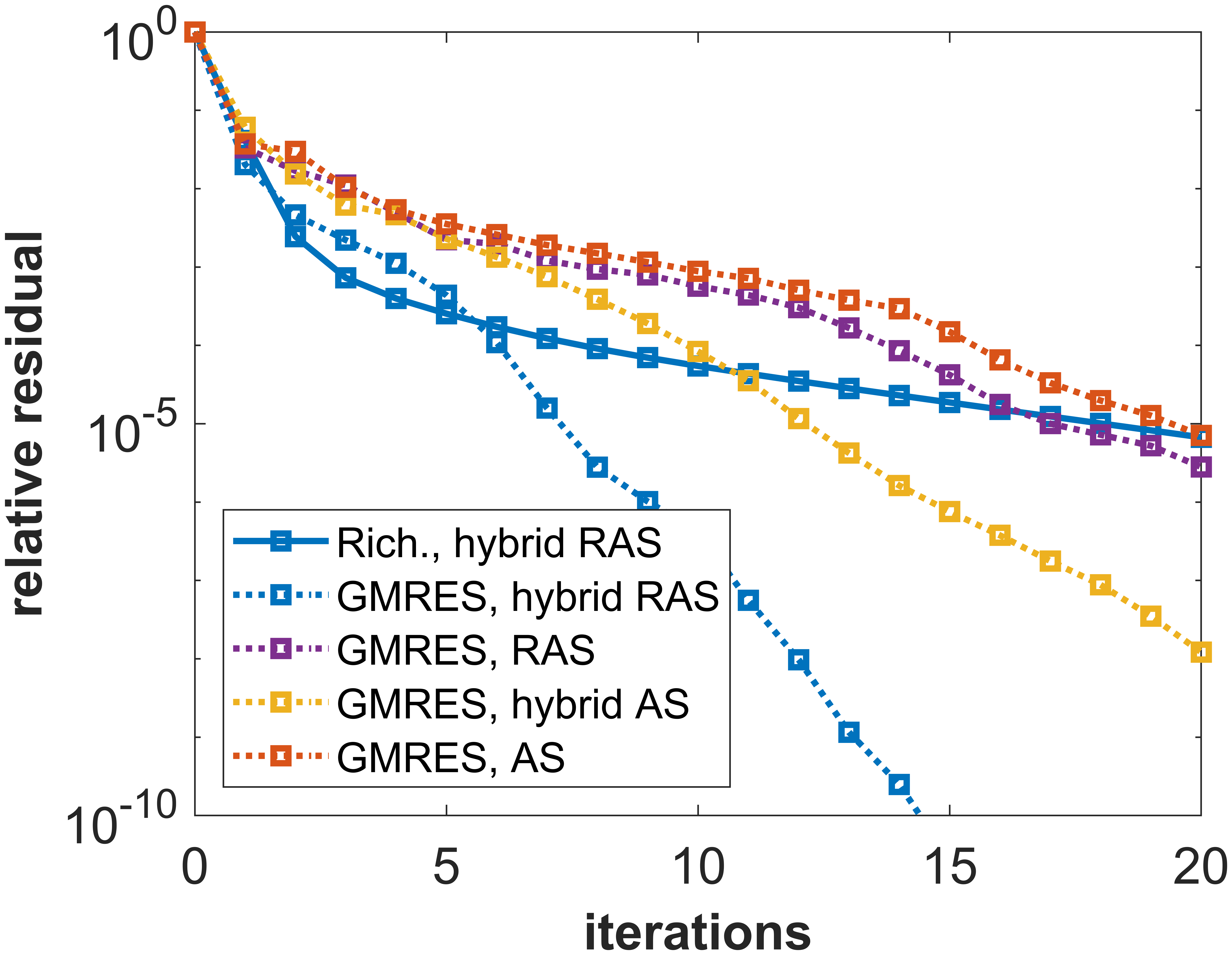}
    \caption{Results for the diffusion example solved with different preconditioner schemes. The plots show the convergence history of these algorithms with the MS-GFEM coarse space (left) and the GenEO coarse space (right), all with a fixed choice of oversampling and local space sizes (Ovsp=8, \#Eig=24). }
\end{figure}


\begin{figure} \label{fig:convergenceRate}
    \includegraphics[width=.48\textwidth]{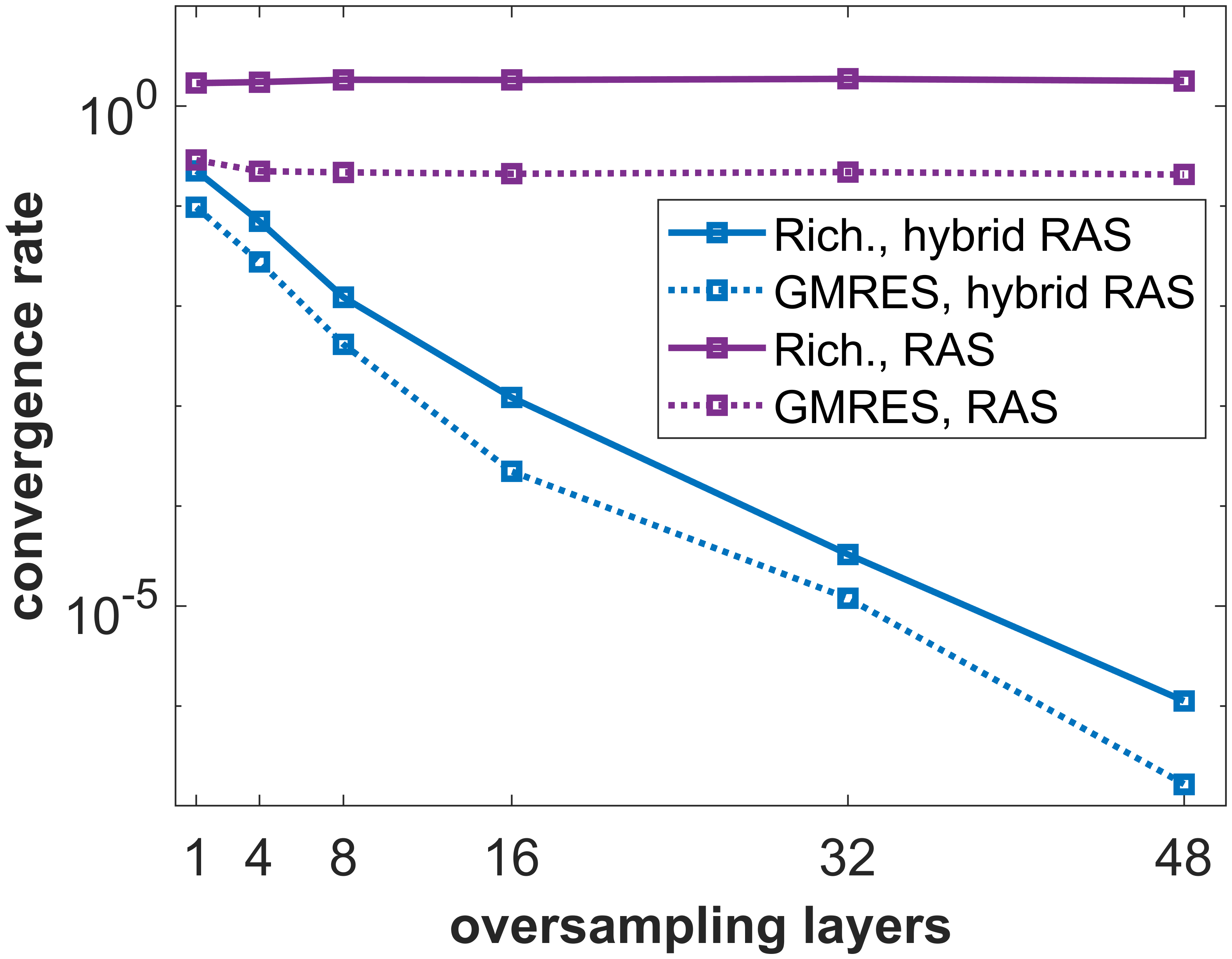}\hfill
    \includegraphics[width=.48\textwidth]{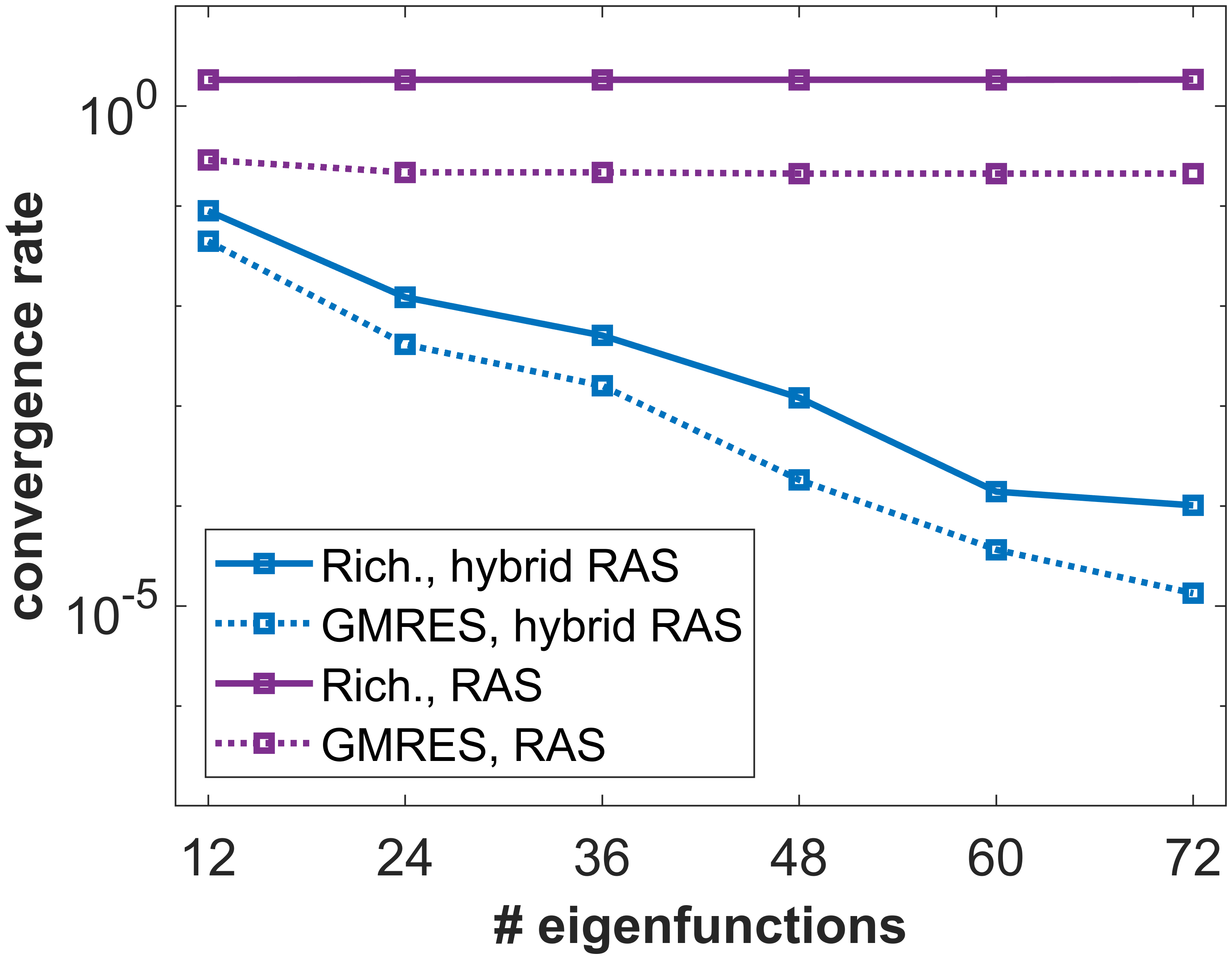}
    \\[\smallskipamount]
    \includegraphics[width=.48\textwidth]{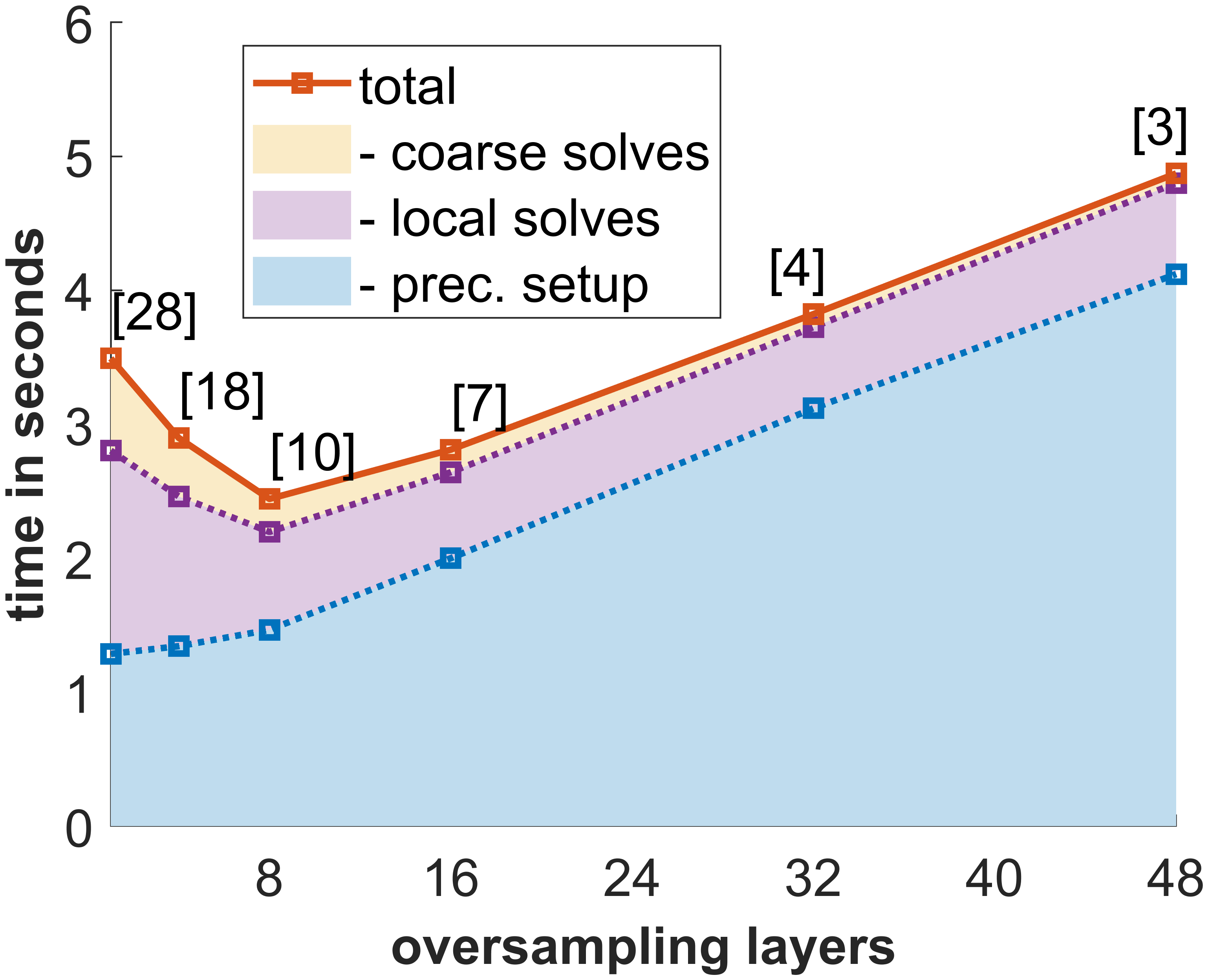}\hfill
    \includegraphics[width=.48\textwidth]{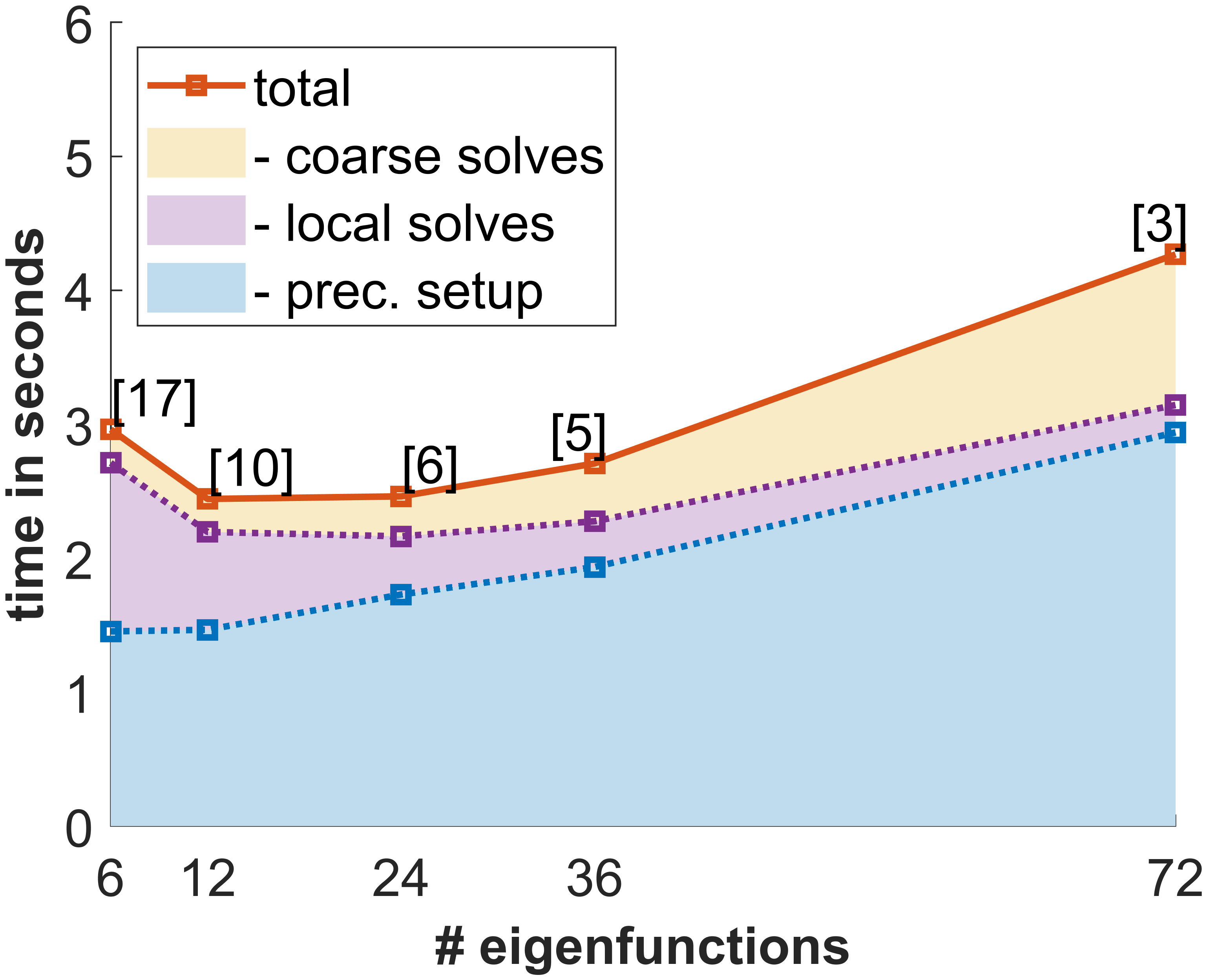}
    \caption{Numerical results for the diffusion example. The top plots show the convergence rates of different MS-GFEM based solvers for increasing oversampling layers (left, with \#Eig = 24 fixed) and numbers of local bases (right, with Ovsp=8 fixed). The bottom plots show the required computational time of the iterative MS-GFEM method as functions of the numbers of oversampling layers (left, with \#Eig = 12 fixed) and local bases (right, with Ovsp=8 fixed). }
\end{figure}

In \cref{heatmaps}, we display the performance of MS-GFEM within the Richardson iterative method for a wide choice of oversampling and local space sizes. It can be clearly seen that as the oversampling size and/or the dimension of the local bases increase, the number of iterations needed for the convergence is significantly reduced, as predicted by our theory. Indeed, we can achieve a convergence rate as high as $10^{-9}$ with a moderate amount of oversampling and local basis functions, whereas increasing the size of the local bases or the oversampling does not significantly improve on this, possibly due to rounding errors. However, the choice of parameters with the fastest convergence rate is not optimal in terms of computation time. This is because the iterative solver itself also becomes more expensive with larger oversampling (local solves) and a larger number of eigenfunctions (coarse solve). Indeed, as shown in \cref{heatmaps} (right), there is a sweet spot in the trade-off between oversampling layers, number of eigenfunctions, and number of iterations to achieve the desired residual reduction in minimal time. In \Cref{convergencePlots}, we further display the performance of several parameter settings in detail. Unsuprisingly the cheapest (most expensive) eigensolves give the slowest (fastest) residual reduction per iteration. The least (total) computational time is achieved here with Ovsp = 8 and \#Eig = 12. While in practice, it is hard to optimally choose these parameters a priori, it is notable that the least computational time is not achieved with the extremes. In particular, it may not be desirable to construct the perfect MS-GFEM coarse space that can solve the problem in one shot, even when solving the same problem repeatedly.

In \cref{fig:compareSolvers} (left), we compare the performance of different two-level Schwarz preconditioners based on the MS-GFEM coarse space with a fixed choice of oversampling and local space sizes. We test four preconditioner schemes -- RAS, AS, hybrid RAS, and hybrid AS, where 'hybrid' means that the coarse space is added multiplicatively. The results show that among the four schemes, the hybrid RAS preconditioner proposed in this paper yields the fastest convergence rate. Moreover, except for the hybrid RAS, the other three lead to divergent Richardson iterative methods (not shown in the plot). As expected, the hybrid RAS preconditioner with the MS-GFEM coarse space converges faster when accelerated with GMRES as opposed to a simple Richardson iteration. In \cref{fig:compareSolvers} (right), we show the performance of the above four preconditioner schemes with the GenEO coarse space. For a fair comparison, here we use the same amount of oversampling for the GenEO eigensolves and local solves, although the original GenEO method was defined without oversampling. We see that each MS-GFEM preconditioner outperforms the corresponding GenEO preconditioner in terms of convergence rate, which demonstrates the importance of defining the local eigenproblems on the harmonic subspaces. Moreover, the convergence rate of the hybrid RAS preconditioner is the fastest, as for the MS-GFEM coarse space.


While \cref{fig:compareSolvers} shows that the standard RAS preconditioner based on MS-GFEM does not perform as well as the hybrid one, it is still unclear that whether increasing the oversampling size or the number of local basis functions can significantly speed up its convergence as for the latter. In \cref{fig:convergenceRate} (top), we show that this is not the case -- its convergence rate remains basically unchanged with increasing oversampling or local space sizes. This observation shows that it is essential to add the MS-GFEM coarse space multiplicatively. In \cref{fig:convergenceRate} (bottom), we plot the required computational time of different parts of the MS-GFEM iterative method as functions of the numbers of oversampling layers (left) and eigenfunctions (right). We see that with increasing the oversampling size, the time of eigensolves increases significantly. Meanwhile, since fewer iterations are needed, the time of local solves remain roughly the same, and the time of the coarse solve is greatly reduced. On the other hand, when increasing the number of local bases, the time of eigensolves increases mildly, but the time of the coarse solve increases significantly, even though the iteration number drops. As in \cref{heatmaps} (right), we clearly see that there are sweet spots, where the sum of the time of eigensolves, local solves and coarse solve are minimised.

\subsection{Example 2: Linear elasticity for composite aero-structures}
In this example, we apply the MS-GFEM preconditioner to three-dimensional linear elasticity equations in composite aero-structures, to evaluate its performance for realistic applications. The structural component that we simulate is a C-shaped wing spar (C-spar) of length 500mm, with a joggle region in its center. The material is a laminated composite with 24 uni-directional layers (plies). Each layer has a thickness of 0.2mm and is made up of carbon fibers embedded into resin. We refer to \cite{benezech2022scalable} for a more detailed description of the model. For the FE discretization, we use piecewise linear elements on a hexahedral mesh with one element through thickness per layer within the DUNE framework \cite{bastian2021dune}. See \Cref{weakScaling} (left) for an illustration of the structure and the FE grid. We run the experiment on the HPC bwForCluster Helix, which offers compute nodes with 64 CPU cores (2x AMD Milan EPYC 7513) and 236GB RAM per node. The preconditioned system is solved using standard GMRES, and the local eigenproblems are solved by the Arpack solver built in DUNE. A (relative) residual reduction of $10^{-6}$ is used as the stopping criterion for the GMRES iteration.

\begin{figure} \label{weakScaling}
    \raisebox{10mm}{\includegraphics[width=0.5\textwidth]{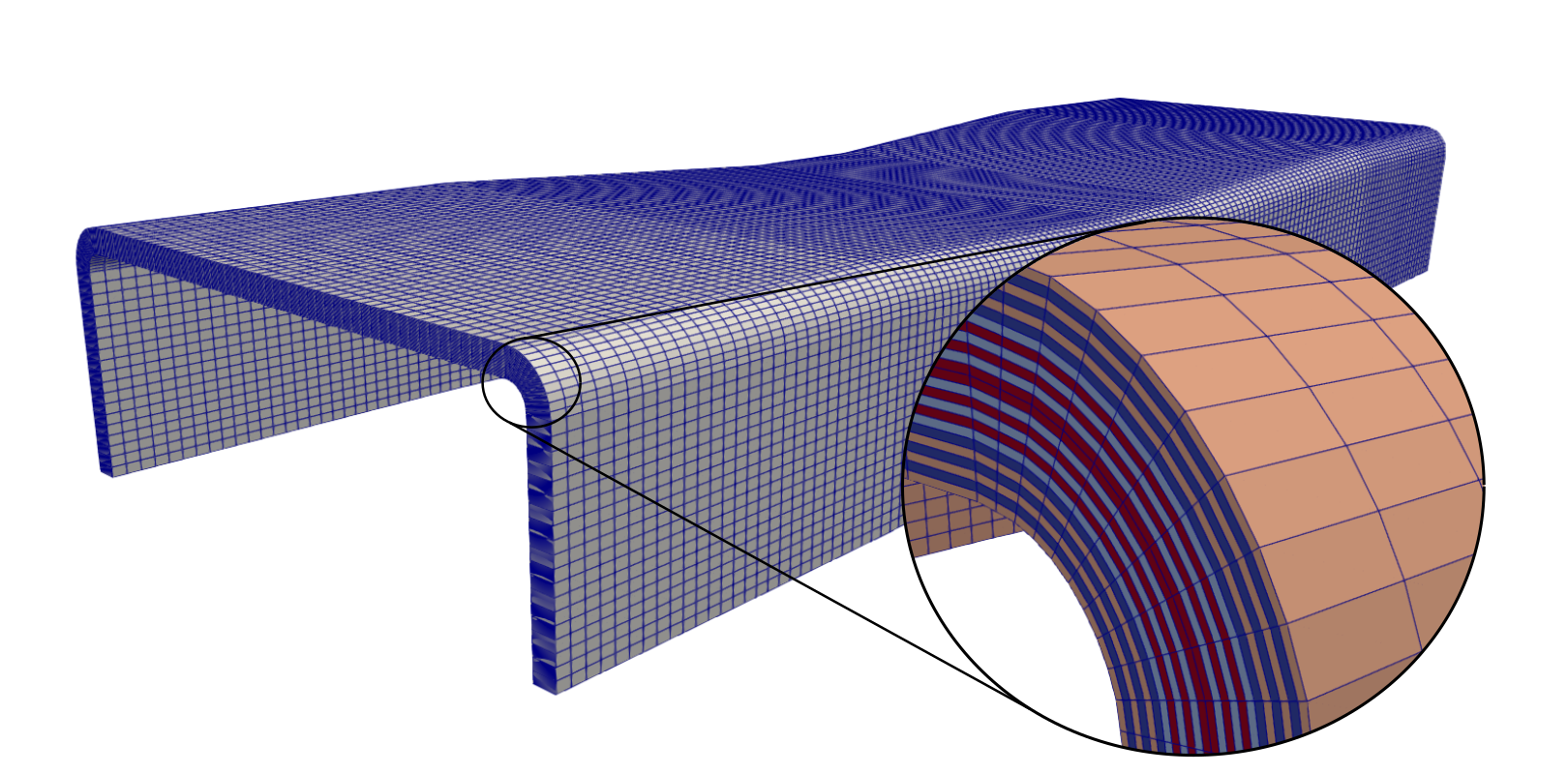}}\hfill
    \includegraphics[width=0.5\textwidth]{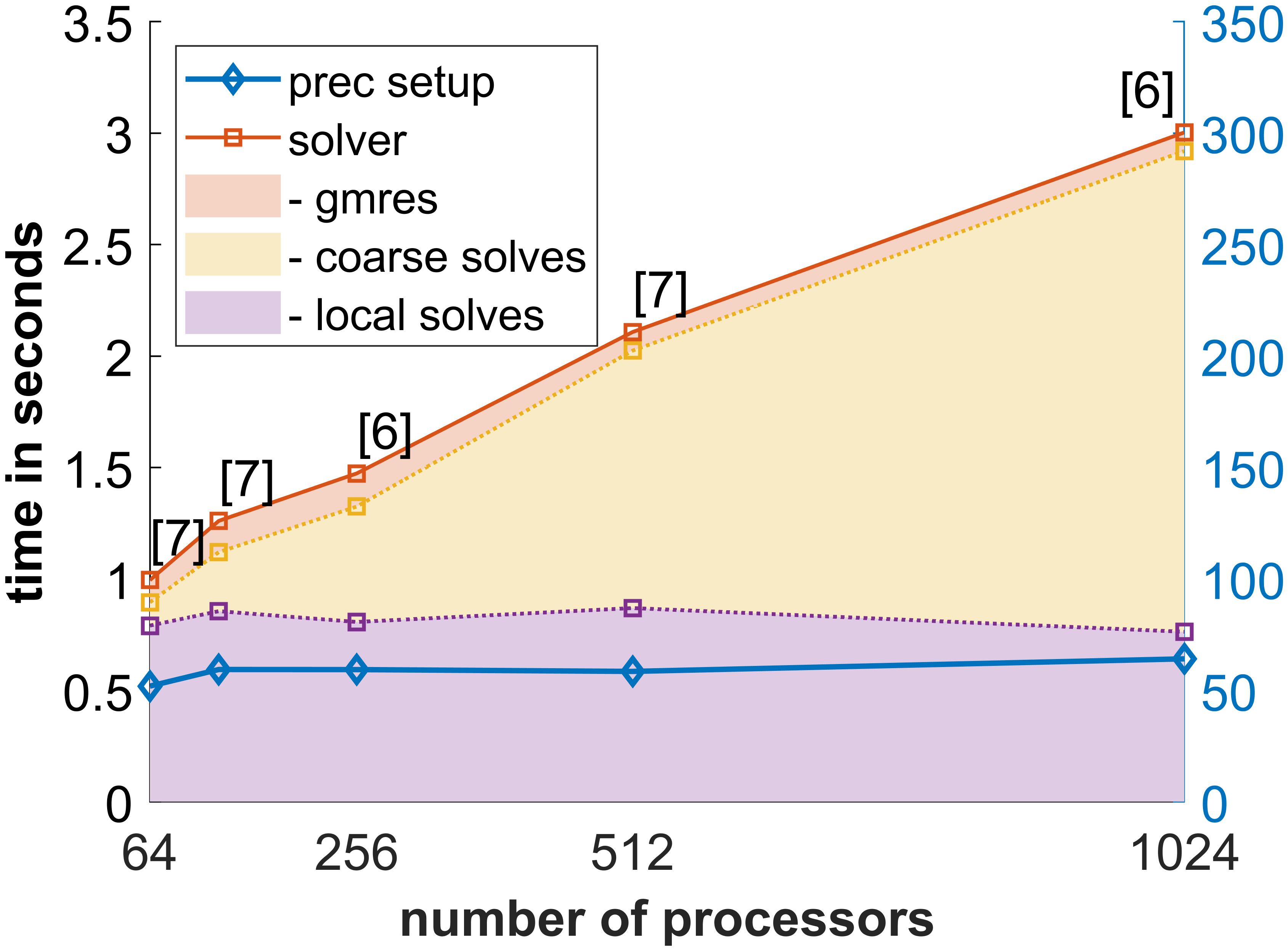}
    \caption{Composite C-spar experiment. Left: an illustration of the C-spar with the layers and the FE grid. Right: results of the weak scaling test for the five models described in \Cref{problemSize}, each with one oversampling layer and 30 eigenfunctions used per subdomain. The numbers in the bracket are the iteration counts. The time of the preconditioner setup ('prec setup') corresponds to the y-axis on the right side.}
\end{figure}

\begin{table}[h!]\label{problemSize}
    \centering
    \begin{tabular}{|r|r|r|}
        \hline
        Length in mm & DOFs & Cores \\
        \hline
        125 & 297,000 & 64 \\
        250 & 585,750 & 128 \\
        500 & 1,163,250 & 256 \\
        1000 & 2,318,250 & 512 \\
        2000 & 4,628,250 & 1024 \\
        \hline
    \end{tabular}
    \caption{Composite C-spar experiment. Details of the models used for the weak scaling test: length in mm, number of DOFs, and number of cores/subdomains used. If not otherwise specified, provided results below are for the model with $\text{length} = 500\text{mm}$ and $\text{cores} = 256$.}
\end{table}

We first perform a weak scaling test for different C-spar models with problem size that scales with the number of processors used. We use five C-spar models with lengths varying from 125mm to 2000mm, and fix the discretization of the mesostructure such that the total number of elements scales with the length. The ratio of the number of DOFs to the number of processors used for each model is kept constant; see \cref{problemSize}. The computational time for the five models is displayed in \cref{weakScaling} (right). We see that the time of the preconditioner setup dominates the overall computational time, and that it scales perfectly with the number of processors. It is because the cost of this part is dominated by the time spent on the local eigensolves that can be performed fully in parallel. Similarly, we observe perfect scaling for the local solves. On the other hand, the time spent on the coarse solves increases significantly for larger models, because the coarse problem is still solved on one processor with direct solvers and its size scales with the number of subdomains (processors).

\begin{figure} \label{heatmapsLinMSGFEM}
    \includegraphics[width=.49\textwidth]{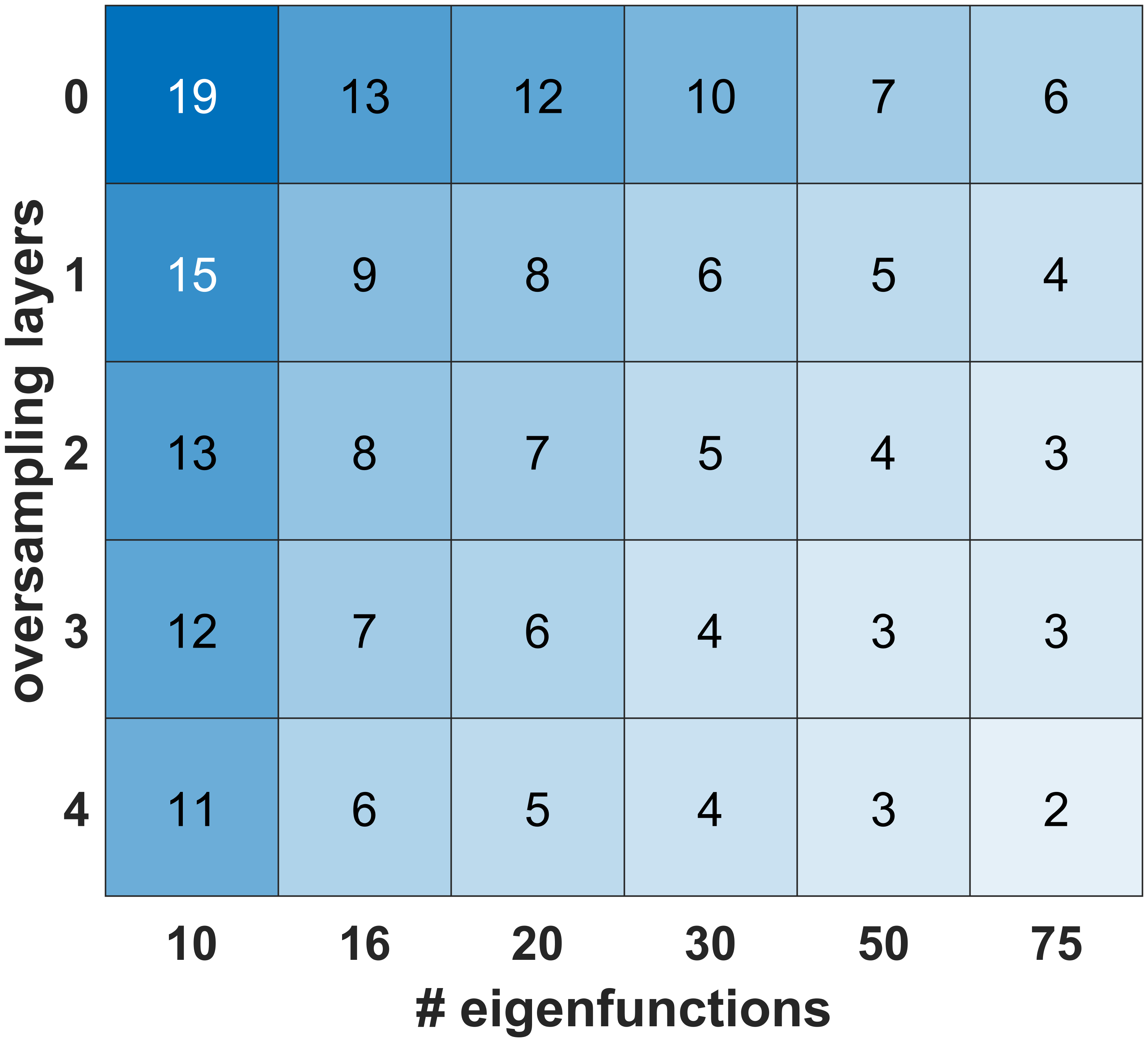} \hfill
     \includegraphics[width=.49\textwidth]{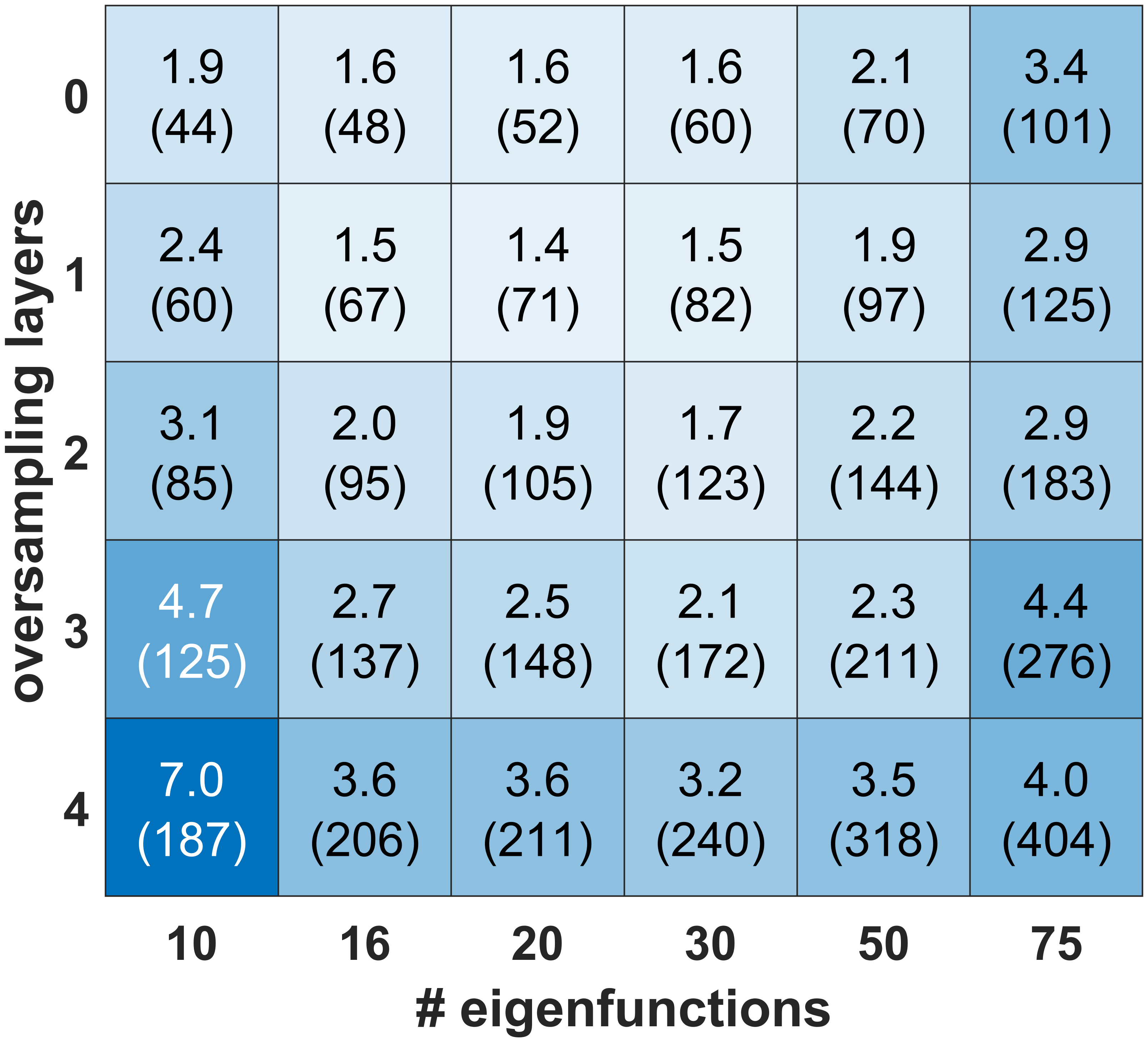}
    \caption{Composite C-spar experiment. Iteration counts (left) and computational time (right) of the preconditioned GMRES for various numbers of eigenfunctions and oversampling layers. On the right, the numbers outside the brackets represent the solver time and the ones in the brackets represent the time of the preconditioner setup.}
\end{figure}

\begin{figure} \label{linearElasticitySimple}
    \includegraphics[width=.48\textwidth]{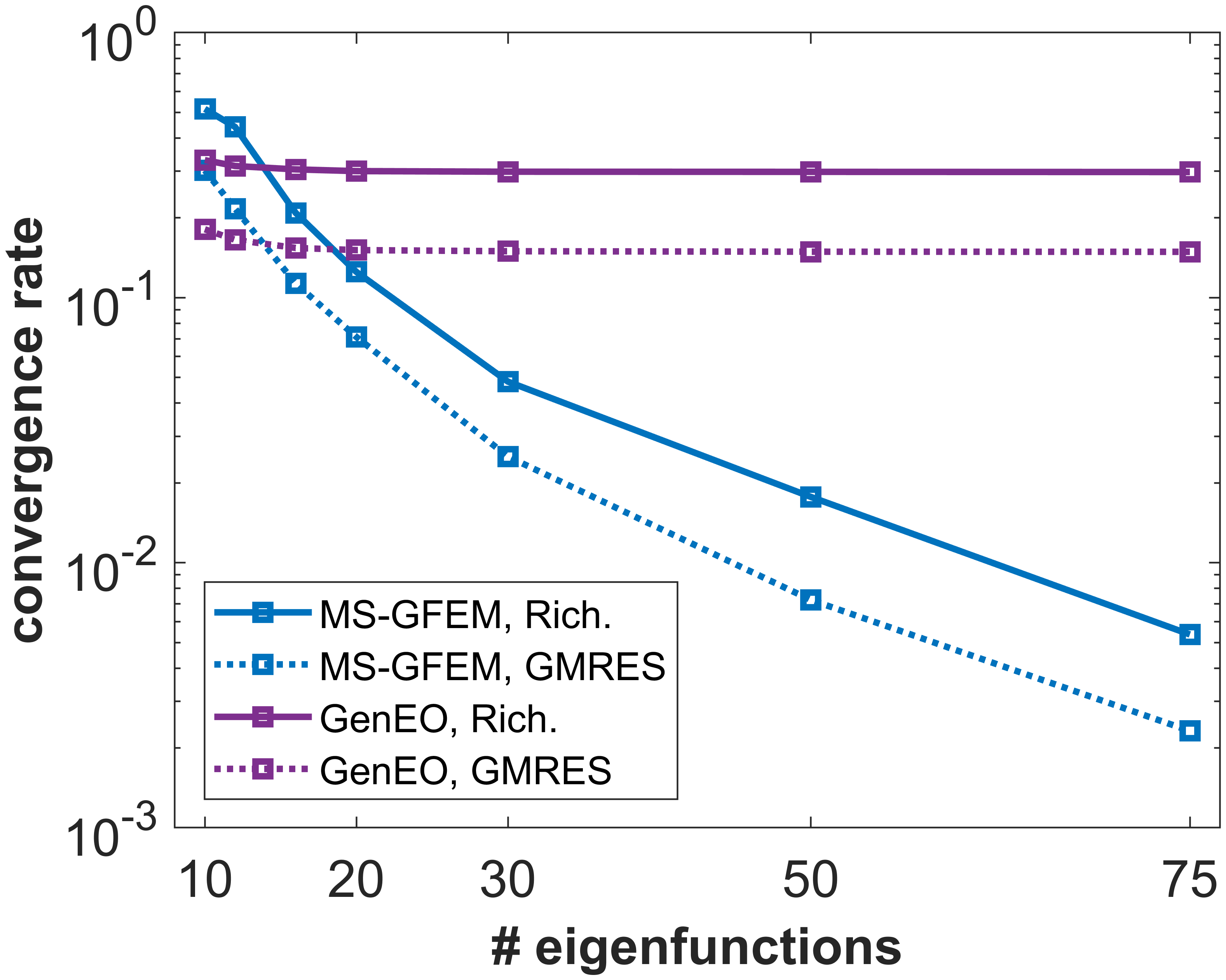}\hfill
    \includegraphics[width=.48\textwidth]{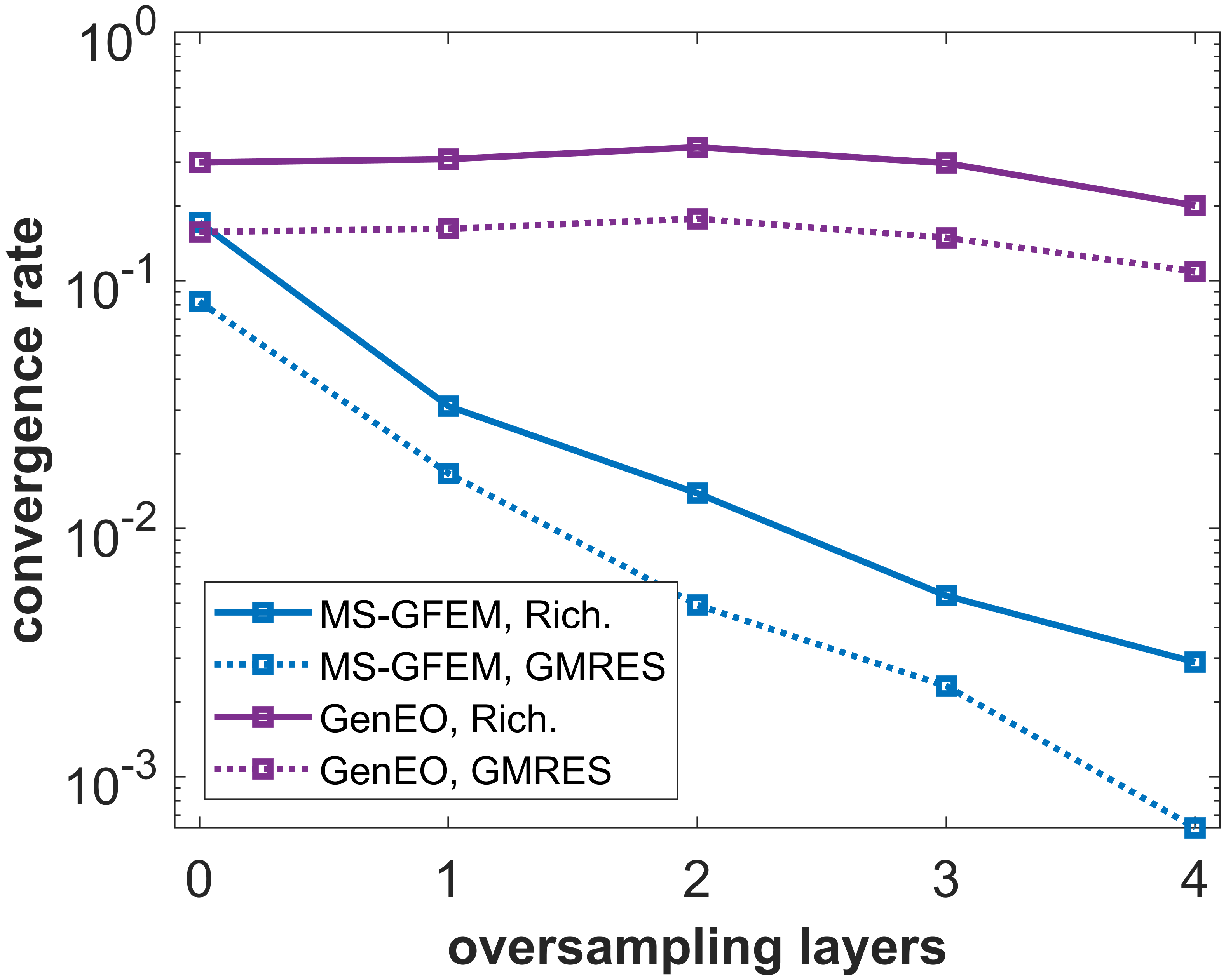}
    \\[\smallskipamount]
    \includegraphics[width=.48\textwidth]{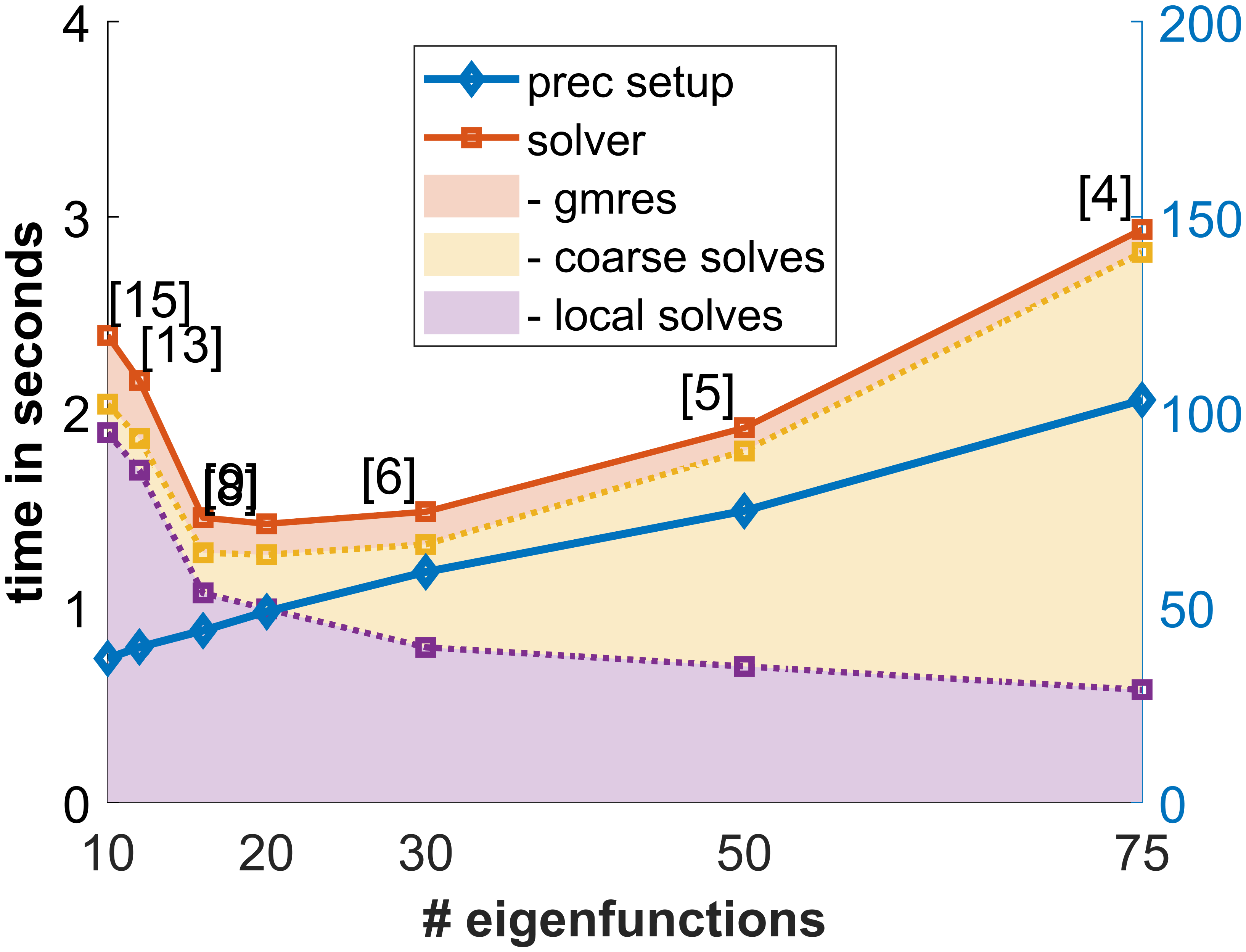}\hfill
    \includegraphics[width=.48\textwidth]{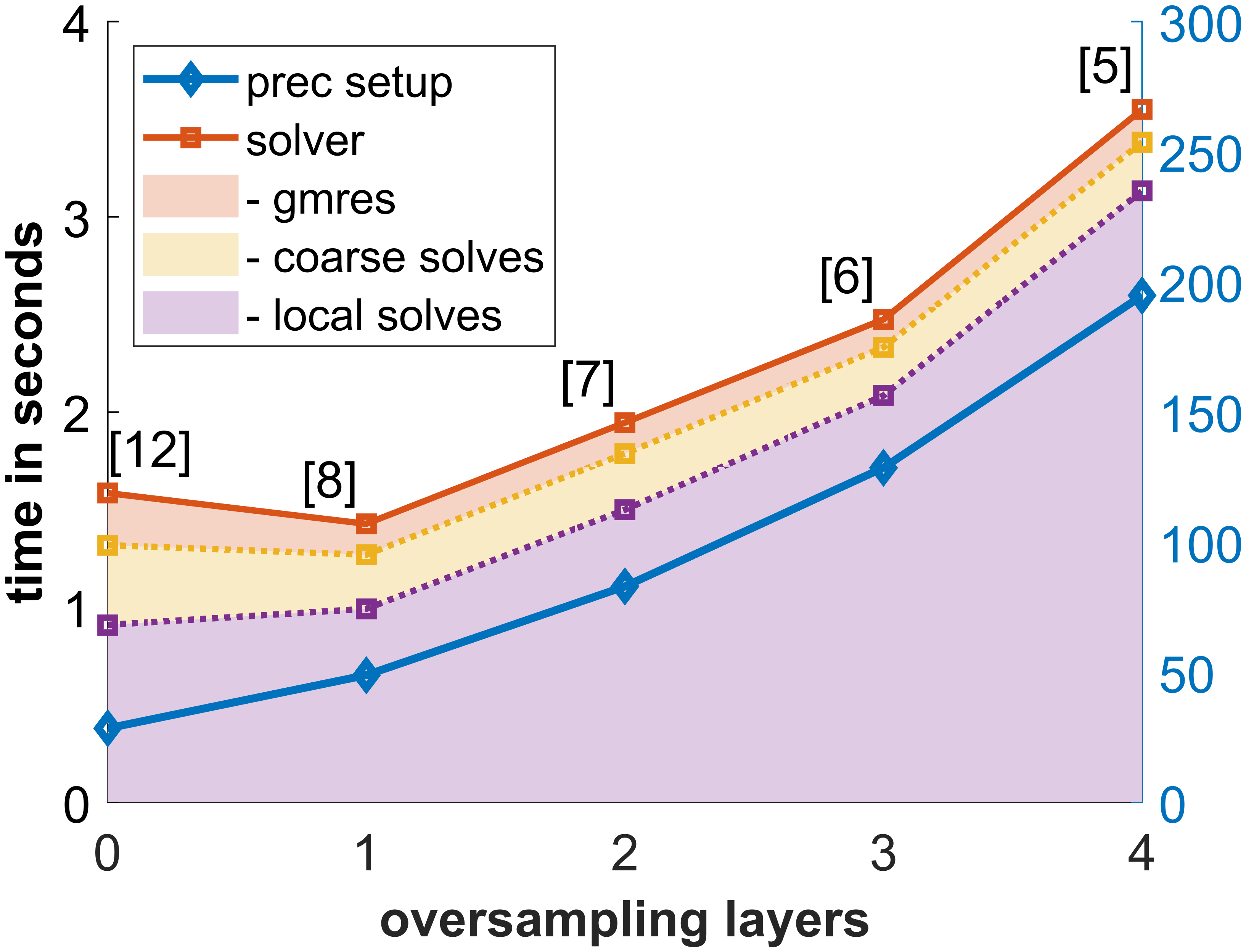}
    \caption{Composite C-spar experiment. Top: convergence rates of the MS-GFEM and GenEO preconditioners, with three oversampling layers used for the left and 75 eigenfunctions used per subdomain for the right. Bottom: Computational time of the preconditioned GMRES, with one oversampling layer used for the left and 20 eigenfunctions used per subdomain for the right.} 
\end{figure}

\Cref{heatmapsLinMSGFEM} gives the iteration counts and computational time for the method with various choices of oversampling and local space sizes. As expected, we observe a significant decrease in the iteration numbers when increasing oversampling and/or the number of local eigenfunctions. This inevitably brings a significant increase in the cost of the preconditioner setup. But as for the two-dimensional example, even ignoring the preconditioner setup, it does not pay to use a large amount of oversampling and local basis functions to minimize the iteration numbers -- the minimal solver time is attained with a small oversampling and a moderate number of local eigenfunctions. For this realistic application, the MS-GFEM preconditioner also outperforms the original method as a 'direct' solver in terms of computational time.

In \cref{linearElasticitySimple} (top), we compare the convergence rates of the MS-GFEM and GenEO preconditioners for increasing numbers of eigenfunctions and oversampling layers. The results clearly show that the convergence rate of the MS-GFEM preconditioner becomes much higher as we increase the number of eigenfunctions or oversampling layers, while that of the GenEO preconditioner remains nearly constant. This agrees well with our theory. \Cref{linearElasticitySimple} (bottom) shows the detailed computational time of the preconditioned GMRES for various choices of oversampling layers and eigenfunctions. We see that as we increase the number of eigenfunctions, the cost of the preconditioner setup grows mildly, but the cost of coarse solves increases significantly, even with reduced iteration numbers. This illustrates again that in practice, it is suitable to use a moderate number of local eigenfunctions. On the other hand, using a larger amount of oversampling greatly increase the cost of the preconditioner setup and of local solves, which can not be compensated for by the savings in the coarse solves. Therefore, in practice, a small size of oversampling is recommended.

\section{Conclusions}
We have formulated MS-GFEM, which was originally proposed as a multiscale discretization method, as an iterative solver and as a preconditioner, and performed a rigorous convergence analysis. Both the iterative solver and the preconditioner for GMRES converge at a rate of the error of the underlying MS-GFEM, and thus can be made 'arbitrarily' fast by means of its superior approximation properties. The theory developed in this paper is very general, applicable to various elliptic problems with highly heterogeneous coefficients. Numerical results for realistic applications confirm our theory and demonstrate that the iterative methods outperform the original MS-GFEM in terms of computational time.

Compared with the original MS-GFEM, the iterative methods allow the use of inexact eigensolvers, thereby providing a trade-off between the accuracy of eigensolves and the amount of work required for them. An important issue is then to determine what level of accuracy is needed to maintain the rapid convergence of the iterative methods. This will be a focus of future work. Moreover, we will extend the theory developed in this paper to indefinite problems, for example, Helmholtz problems with large wavenumbers.

\section{Acknowledgements}
The authors acknowledge support by the state of Baden-Württemberg through bwHPC and the German Research Foundation (DFG) through grant INST 35/1597-1 FUGG and through its Excellence Strategy EXC 2181/1 - 390900948 (the Heidelberg STRUCTURES Excellence Cluster).

\bibliographystyle{siamplain}
\bibliography{lit}

\end{document}